\documentclass[english,a4paper,oneside,11pt]{amsart}
\usepackage{amssymb,amscd,mathrsfs,amsbsy}
\usepackage[T1, T2A]{fontenc}
\usepackage[utf8]{inputenc}
\usepackage[english]{babel}
\usepackage[shortlabels]{enumitem}

\usepackage{concrete}
\usepackage{xcolor}
\usepackage{xspace}
\usepackage{mathtools}
\usepackage{amsthm}
\usepackage{moreenum}
\usepackage{asymptote}
\usepackage{graphicx}
\usepackage{breakcites}
\usepackage{phoenician}

\usepackage{stmaryrd}
\usepackage{etoolbox}
\usepackage{quiver}
\usepackage[cmtip,arrow]{xy}

\usepackage{pb-diagram,pb-xy}
\usepackage{mathtools}

\usepackage{extarrows}

\usepackage{float}

\usepackage{csquotes}
\usepackage{tikz}
\usetikzlibrary{patterns}
\usepackage{tikz-cd}
\usepackage{dutchcal}
\usepackage{epigraph}
\usepackage[hyperindex,breaklinks]{hyperref}
\hypersetup{
	colorlinks,
	citecolor=blue,
	filecolor=black,
	linkcolor=blue,
	urlcolor=black
}
\usepackage[a4paper]{geometry}

\newcommand{\ul}[1]{\underline{#1}}
\newcommand{\mb}[1]{\boldsymbol{#1}}
\newcommand{\ol}[1]{\overline{#1}}
\newcommand{\mc}[1]{\mathcal{#1}}
\newcommand{\ms}[1]{\mathscr{#1}}
\newcommand{\mf}[1]{\mathfrak{#1}}
\newcommand{\wh}[1]{\widehat{#1}}
\newcommand{\wt}[1]{\widetilde{#1}}

\addto\captionsenglish{}
\theoremstyle{definition}
\newtheorem{definition}{\ul{Definition}}[section]
\newtheorem{proposition}[definition]{\ul{Proposition}}
\newtheorem{lemma}[definition]{\ul{Lemma}}

\newtheorem{corollary}[definition]{\ul{Corollary}}
\newtheorem{example}[definition]{\ul{Example}}
\newtheorem{construction}[definition]{\ul{Construction}}
\theoremstyle{theorem}
\newtheorem{theorem}[definition]{\ul{Theorem}}
\newtheorem*{theorem*}{\ul{Theorem}}
\theoremstyle{remark}
\newtheorem*{remark}{\ul{Remark}}
\newtheorem*{notation}{\ul{Notation}}

\def\R{\mathbb{R}}

\def\D{\mc{D}}
\def\As{\ms{A}}
\def\Ab{\mc{Ab}}
\def\Bs{\ms{B}}
\def\dl{\mc{dl}}
\def\dCart{\mc{dCart}}
\def\dgMan{\mc{dgMan}}
\def\dgCart{\mc{dgCart}}
\def\CL{\mc{CL}}
\def\MM{\mathbb{M}}
\def\NN{\mathbb{N}}

\def\U{\ms{U}}

\def\Loc{\mc{Loc}}
\def\Alg{\mc{Alg}}
\def\s{\mc{s}}
\def\dg{\mc{dg}}
\def\Com{\mc{Com}}

\def\pd{\partial}
\def\II{\mathbb{I}}

\def\Tc{\mc{T}}
\def\Qc{\mc{Q}}
\def\A{\mc{A}}
\def\B{\mc{B}}
\def\N{\mc{N}}
\def\EZ{\mc{EZ}}

\DeclareMathOperator{\Zar}{Zar}

\DeclareMathOperator{\Der}{Der}

\DeclareMathOperator{\Ho}{Ho}
\DeclareMathOperator{\Mod}{Mod}
\DeclareMathOperator{\Tor}{Tor}
\DeclareMathOperator{\Nat}{Nat}

\DeclareMathOperator{\rank}{rank}

\DeclareMathOperator{\Vect}{Vect}
\DeclareMathOperator{\gr}{gr}
\DeclareMathOperator{\Ch}{Ch}
\DeclareMathOperator{\Cart}{Cart}
\DeclareMathOperator{\id}{id}
\def\F{\mc{F}}
\def\y{\mc{y}}

\def\O{\mc{O}}
\def\Cc{\mc{C}}
\def\FF{\mathbb{F}}

\def\ve{\varepsilon}

\def\Rc{\mc{R}}

\def\SS{\mathbb{S}}
\def\sse{\subseteq}
\DeclareMathOperator{\Cov}{Cov}
\DeclareMathOperator{\op}{{op}}
\DeclareMathOperator*{\colim}{colim}
\DeclareMathOperator{\PSh}{PSh}

\DeclareMathOperator{\Spec}{Spec}
\DeclareMathOperator{\Specm}{Specm}
\DeclareMathOperator{\Hom}{Hom}

\DeclareMathOperator{\Sh}{Sh}

\def\Z{\mathbb{Z}}
\def\C{\mathbb{C}}

\def\M{\ms{M}}
\def\Ns{\ms{N}}
\def\dMan{\mc{dMan}}
\def\L{\mc{L}}

\def\AA{\mathbb{A}}
\def\la{\langle}
\def\ra{\rangle}
\def\UU{\mathbb{U}}
\def\LL{\mathbb{L}}
\def\K{\mc{K}}
\def\sSet{\mc{sSet}}
\def\Set{\mc{Set}}
\def\CI{\CcI}
\def\Mc{\mc{M}}
\def\sCAlg{\mc{sCAlg}}
\def\DGCAlg{\mc{dgCAlg}}

\def\CcI{\Cc^\infty}
\def\W{\mc{W}}

\def\l{\mc{l}}
\def\II{\mathcal{II}}

\def\AKSZ{Alexandrov--Kontsevich--Schwarz--Zaboronsky\xspace}

\newcommand{\BLX}{Behrend--Liao--Xu\xspace}
\def\TV{To\"{e}n--Vezzosi\xspace}
\def\DHI{Dugger--Hollander--Isaksen\xspace}
\def\AGV{Artin--Grothendieck--Verdier\xspace}
\def\DHKS{Dwyer--Hirschhorn--Kan--Smith\xspace}
\def\doi#1{\href{https://doi.org/#1}{doi:#1}}
\def\arXiv#1{\href{https://arxiv.org/abs/#1}{arXiv:#1}}
\def\sfdg{\mc{sfdg}}
\def\MR{Moerdijk--Reyes\xspace}

\begin{document}
	\title[Equivalent models of derived stacks]{Equivalent models of derived stacks}
	\author{Gregory Taroyan}
	\maketitle
	\pagenumbering{arabic}
	\begin{abstract}
		In the present paper, we establish an equivalence between several models of derived geometry. That is, we show that the categories of higher derived stacks they produce are Quillen equivalent. As a result, we tie together a model of derived manifolds constructed by Spivak--Borisov--Noel, a model of Carchedi--Roytenberg, and a model of \BLX. By results of \BLX the latter model is also equivalent to the classical \AKSZ model. This equivalence allows us to show that weak equivalences of derived manifolds in the sense of \BLX correspond to weak equivalences of algebraic models of these manifolds, thus proving the conjecture of \BLX. Our results are formulated in the framework of Fermat theories, allowing for a simultaneous treatment of differential, holomorphic, and algebraic settings. 
	\end{abstract}
	\tableofcontents
	\section{Introduction}
	Derived geometry in its modern form appeared as the culmination of a program initiated by Quillen \cite{quillen1970co}. The main goal of this approach is to study algebraic and algebro-geometric objects and their deformations {\it homotopically}. In this direction, T\"{o}en--Vezzosi~\cite{toen2005homotopical} and Lurie~\cite{lurie2004derived} developed a uniform approach to derived algebraic geometry using homotopy theory of higher sheaves.
	\par An approach to differential geometry via \(\CI\)-rings is a powerful method allowing to apply purely algebraic and categorical techniques to the smooth setting. A classical account is due to Dubuc \cite{dubuc1981c}, for a modern textbook treatment see Moerdijk--Reyes~\cite{moerdijk2013models}. Using this setting, Spivak constructed a model for derived differential geometry following the ideas of Lurie in \cite{spivak2010derived}. Spivak's construction used homotopy sheaves of homotopical \(\CI\)-rings. Borisov--Noel \cite{borisov2011simplicial} simplified Spivak's construction by replacing homotopy~\mbox{\(\CI\)-rings} with simplicial \(\CI\)-rings. Around the same time, Carchedi--Roytenberg~\cite{carchedi2012theories},~\cite{carchedi2012homological} proposed a different approach to derived differential geometry based on \emph{differential graded \(\CI\)-rings} rather than simplicial ones. The importance of derived differential geometry can't be stressed enough since it provides a uniform treatment for such important objects as DG-manifolds (Q-manifolds) originally due to Schwarz and Kontsevich \cite{schwarz1993semiclassical},\;\cite{alexandrov1997geometry}, derived critical loci (see, e.g., Vezzosi \cite{vezzosi2011derived}), and others. It also allows for calculations with various flavours of cobordisms even in settings where there are no suitable transversality results (for example, in equivariant topology), see \cite[Theorem 3.12]{spivak2010derived}. Thus it is interesting to be able to choose models of derived differential geometry better suited for calculations. 
	\par The problem of equivalence of Spivak--Borisov--Noel's and Carchedi--Roytenberg's approach remained unresolved for some time. Nuiten \cite[Corollary 2.2.10]{nuiten2018lie} used techniques developed by Lurie \cite[\S 5.5]{lurie2009higher} to prove that the \((\infty,1)\)-category of DG \(\CI\)-rings is equivalent to \((\infty,1)\)-category of functors \(\CI\to \sSet\) that preserve products up to homotopy. These can be considered to be analogous to homotopical \(\CI\)-rings of Spivak, although the construction is not quite the same. Nuiten's approach, however, does not yield a direct equivalence; instead, it reduces the case of \(\CI\)-rings to the case of commutative algebras, utilizes Dold--Kan correspondence for them and then recreates categories of simplicial and DG \(\CI\)-rings as the closure under homotopy sifted colimits of free objects. Pridham \cite[Proposition 2.13]{pridham2020differential} remarks that Nuiten's result can be extended to algebras over arbitrary Fermat theories over fields of characteristic \(0.\)
	\par In the present text, however, we explicitly construct a Quillen equivalence between categories of simplicial and DG algebras for a Fermat theory \(\Tc\) over a field of characteristic \(0.\) The key idea is to extend Quillen's original approach to the monoidal Dold--Kan correspondence for commutative algebras sketched in \cite[Remark on p.223]{quillen1969rational} for an arbitrary Fermat theory \(\Tc\). As a result, we show that the functor of normalized chains and its left adjoint, which we explicitly describe, yield a Quillen equivalence between the model categories of simplicial and DG algebras over \(\Tc.\) A pleasant aside of our work is a "hands-on" proof of the Dold--Kan correspondence for commutative algebras following Quillen's original ideas, which to our knowledge is not present in the literature.
	\par We then use this version of the Dold--Kan correspondence to establish the equivalence between simplicial and DG models for derived synthetic differential geometry (Theorem~\ref{thm:eq_for_models}). This result applies in even greater generality and provides a similar equivalence for synthetic analytic geometry in general (Theorem \ref{thm:eq_for_models_gen}). Moreover, we show that derived geometry can be done not only both on simplicial and DG sites for an arbitrary Fermat theory \(\Tc\) but on even simpler {\it Cartesian site} \(d\Cart_{\Tc}\) spanned by the duals of semifree algebras. Hence, we can uniformly define such objects as principal bundles with connection, higher gerbes and so on for all stacks by specifying their behaviour only on this simple site. 
	\par Finally, we establish an equivalence of our approach to derived differential geometry with the theory of derived manifolds developed by Behrend--Liao--Xu \cite{behrend2020derived}. As a part of this equivalence result, we show that our theory is equivalent to the approach via non-negatively graded DG (Q) manifolds of Schwarz \cite{schwarz1993semiclassical} and \AKSZ \cite{alexandrov1997geometry}. We also prove a conjecture of Behrend--Liao--Xu \cite[p.6]{behrend2020derived} about the relationship of weak equivalences between derived manifolds and quasi-isomorphisms of their function algebras.  
	\subsection{Organization of the paper and main results}
	The key technical result of this paper is the following version of the Dold--Kan correspondence for Fermat theories:
	\begin{theorem}[{Theorem \ref{thm:DK_for_Fermat}}]\label{thm:intro_1}
		Denote the normalized chains functor from simplicial vector spaces over a field \(K\) of characteristic \(0\) to chain complexes over \(K\) by \(N.\) Let \(\Tc\) be a Fermat theory over \(K\) (Definition \ref{def:Fermat_over_R}). Then \(N\) induces a right Quillen equivalence between categories of simplicial and DG algebras over \(\Tc.\)
	\end{theorem}
	Theorem \ref{thm:intro_1} implies the following equivalence result:
	\begin{theorem}[{Theorem \ref{thm:site_eq}}]\label{thm:intro_eq}
		Let \(K\) be a field of characteristic \(0.\) Let \(\Tc\) be a Fermat theory over \(K\) (Definition \ref{def:Fermat_over_R}). Then there are Quillen equivalences between model topoi of sheaves over sites of simplicial, DG, and semifree DG algebras over \(\Tc.\)
	\end{theorem}
	A consequence of Theorem \ref{thm:intro_eq} is the following equivalence result for \(\CI\)-rings.
	\begin{theorem}[{Theorem \ref{thm:eq_for_models}}]
		DG, semifree DG, and simplicial versions of all sites for the smooth infinitesimal analysis described in \cite[Appendix 2]{moerdijk2013models} produce Quillen equivalent sheaf topoi. 
	\end{theorem}
	We also have the following equivalence result between categories of higher sheaves on derived manifolds and dg and categories of higher sheaves on the smooth derived Cartesian site \(d\Cart_{\CI}.\)
	\begin{theorem}[{Theorem \ref{thm:eq_dMan}, Proposition \ref{pro:equiv_dg_Cart}}]
		The topos of higher sheaves on the relative site \(\dMan\) of derived manifolds in the sense of Behrend--Liao--Xu \cite{behrend2020derived} is Quillen equivalent to the topos of sheaves on the derived Cartesian site \(d\Cart_{\CI}.\) This topos is also Quillen equivalent to the higher sheaf topos over the site of non-negatively graded dg manifolds in the sense of \AKSZ \cite{alexandrov1997geometry}.
	\end{theorem}
	Finally, we resolve a conjecture of Behrend--Liao--Xu \cite[p. 6]{behrend2020derived} in the affirmative. While we were preparing this paper for uploading to arXiv an alternative proof of Corollary~\ref{cor:conj} was given by Carchedi \cite[Theorem 5.20]{carchedi2023derived}.
	\begin{proposition}[{Corollary \ref{cor:conj}}]
		A morphism of derived manifolds is a weak equivalence if and only if it induces a quasi-isomorphism on algebras of functions.
	\end{proposition}
	Here is an outline of the paper's contents. Section \ref{sec:background} contains the background material about homotopical algebraic geometry and Fermat theories. Section \ref{sec:conn_thm} contains a technical result known as Curtis' connectivity theorem, which we generalize to arbitrary Fermat theories (in any characteristic). Sections \ref{sec:DK_com}, \ref{sec:DK_Fermat} contain proofs of Dold--Kan equivalence for commutative algebras and algebras over an arbitrary Fermat theory in characteristic~\(0.\) This is the technical core of the paper. Section \ref{sec:models} introduces three geometric models for which we prove the equivalence of sheaf topoi: simplicial, DG, and semifree DG algebras over a Fermat theory. Section \ref{sec:main} contains the main result of this paper, Theorem \ref{thm:site_eq}, which establishes an equivalence between model topoi of higher sheaves for all three models from Section \ref{sec:models}.
	\par Section \ref{sec:apps} contains applications of Theorem \ref{thm:site_eq} to concrete Fermat theories. In \S\ref{subsec:eq_synth}, we show that all models of smooth infinitesimal analysis of Moerdijk--Reyes \cite[Appendix 2]{moerdijk2013models} can be promoted to simplicial, DG, and semifree DG models and define the same (up to Quillen equivalence) categories of sheaves. In \S\ref{subsec:eq_dg}, we show that sites of non-negatively graded dg manifolds and \(d\Cart_{\CI}\) produce Quillen equivalent topoi of higher sheaves. \S\ref{subsec:eq_der} establishes the equivalence of our approach to derived differential geometry with that of Behrend--Liao--Xu \cite{behrend2020derived} (Theorem \ref{thm:eq_dMan}). In particular, we prove that weak equivalence between derived manifolds in both senses coincide (Corollary \ref{cor:conj}). Finally, Section \ref{sec:dif_coh} contains a straightforward definition of differential cohomology for derived differentiable stacks. 
	\subsection*{Acknowledgements}
	First, we want to express our deep gratitude to our advisor Dmitri Pavlov for posing the problem and explaining how to solve it. Second, we want to thank our partner Alisa Chistopolskaia for providing constant moral support. Finally, we want to thank Sander Kupers for discussing Quillen's "Rational Homotopy Theory" with us. 
	\section{Background}\label{sec:background}
	\subsection{Homotopical algebraic geometry}
	In this section we fix the notation and introduce key concepts of homotopical algebraic geometry following To\"{e}n--Vezzosi \cite{toen2005homotopical} and Dugger--Hollander--Isaksen \cite{dugger2004hypercovers}. 
	
	\begin{definition}\label{def:coverage}
		{\it A coverage} on a category \(C\) consists of a function assigning to each object \(X\in C\) a collection of families of morphisms \(\{f_i:X_i\to X\}\) called {\it covering families} such that
		\begin{itemize}
			\item If \(\{f_i:X_i\to X\}_{i\in I}\) is a covering family and \(g:Y\to X\) is a morphism, then there exists a covering family \(\{h_j:Y_j\to Y\}\) such that each composite morphism~\(gh_j\)~factors through some \(f_i\) (see the diagram below).
			\[\begin{tikzcd}
				{Y_j} && {X_i} \\
				\\
				Y && X
				\arrow["k", from=1-1, to=1-3]
				\arrow["{h_j}"', from=1-1, to=3-1]
				\arrow["g"', from=3-1, to=3-3]
				\arrow["{f_i}", from=1-3, to=3-3]
			\end{tikzcd}\]
		\end{itemize}
	\end{definition}
	\begin{definition}
		{\it A homotopical coverage} on a relative category is a coverage on the localized category $S^{-1}C.$ A relative category equipped with a homotopical coverage will be referred to as a {\it relative site}.
	\end{definition}
	\begin{definition}\label{def:presheaves}
		Let $(C,S)$ be a small relative category. The model category $\sSet^{C^{\op}}$ (with either injective or projective model structure\footnote{We don't distinguish between these in our notation because the resulting model categories as well as their Bousfield localizations that we define below are Quillen equivalent. This convention is justified by Lemma \ref{pro:inj_proj_equiv}.}) is called the {\it model category of simplicial presheaves} on \((C,S).\) We denote it by $\PSh(C).$
		\par The left Bousfield localization of the category \(\PSh(C)\) with respect to all maps \(\y(s), s\in S_1,\) where \(\y\) is the Yoneda embedding will be referred to as {\it category of pre-stacks on \((C,S).\)} We denote it by \((C,S)^\wedge.\)
	\end{definition}
	To avoid size issues, we introduce the following definition. 
	\begin{definition}[{To\"{e}n--Vezzosi \cite[Definition A.1.1, Definition  4.1.1]{toen2005homotopical}}]\label{def:(pseudo)MC} For a universe~\(\U,\) a \(\U\)-category (see the book of Artin--Grothendieck--Verdier \cite{artin1973theorie}) is a category with all sets of morphisms belonging to \(\U\) and the set of objects being a subset of \(\U.\) A {\it \(\U\)-small model category} \(M\) is a category with a closed model structure in the sense of Hovey \cite[Definition 1.1.3]{hovey2007model} such that \(M\) is a \(\U\)-category and there are all \(\U\)-small limits and colimits.
	\end{definition}
	\begin{remark}
		In what follows we refer to \(\U\)-small categories as {\it small} categories, meaning that the categories in question are \(\U\)-small for some universe~\(\U.\) 
	\end{remark}
	In order to describe the sheaf condition on relative sites, we will need the following definition. 
	\begin{definition}[{\TV \cite[Definition 4.4.1]{toen2005homotopical}, \DHI \cite[Definition 4.2]{dugger2004hypercovers}}] Let \((C,S)\) be a small relative site with a coverage \(\K\). For an object \(x \in C\) denote by \(\Cov_K(x)\) the category of {\it homotopical covers} for \(x,\) i.e., families \(\{U_i\to x\}\) that become a covering family in \((S^{-1}C,\K).\) A simplicial object \(\F_\bullet\) in the category \(\Cov_\K(x)\) becomes a simplicial presheaf over \(\y(x)\) via the Yoneda embedding:
		\[
		\y(\F)_n=\bigsqcup_{U\in \F_n} \y(U).
		\]  
		Such \(\F_\bullet\) is referred to as a {\it homotopical hypercover} over \(x.\)
	\end{definition}
	
	\begin{definition}[{\cite[Corollary 4.6.2]{toen2005homotopical}}]\label{def:sheaves} Let \((C,S)\) be a small relative category. If \((C,S)\) carries a homotopical coverage \(\K\) the left Bousfield localization of \((C,S)^\wedge\) in the class
		\[
		\L=\{\y(\F_\bullet)\to \y(x)\mid x\in C\}
		\]
		where \(\F_\bullet\to \y(x)\) is a homotopical hypercover will be referred to as the {\it category of stacks on \((C,S)\)} and will be denoted by \(\Sh((C,S);\K).\) By a result of \TV \cite[Theorem 4.6.1]{toen2005homotopical} this localization exists. 
	\end{definition}
	\begin{remark}
		A classical account of descent for hypercover on sites can be found in the book of \AGV \cite[F. 1, E. 5, App.]{artin1973theorie} and the book of Artin--Mazur~\mbox{\cite[\S 8]{artin2006etale}}. A more modern exposition relevant for homotopy theory can be found in \DHI \cite[\S 4 onwards]{dugger2004hypercovers} and \TV \cite[\S 4.4]{toen2005homotopical}.
	\end{remark}
	Now we show that the choice of either injective or projective model structure\footnote{The same result with essentially the same proof holds for arbitrary intermediate model structures on presheaves in the sense of Jardine \cite{jardine2006intermediate}.} on the category of prestacks is irrelevant to the definition of sheaf topoi. 
	\begin{lemma}\label{pro:inj_proj_equiv}
		Consider the left Bousfield localization \(L_{\L} (C,S)^\wedge_{proj}\) of the prestack category with the projective model structure with respect to \(\L\) (see Definition \ref{def:sheaves}) and consider the left Bousfield localization \(L_{\L} (C,S)^\wedge_{inj}.\) 
		The identity functors induce a Quillen equivalence between the following model categories
		\begin{enumerate}
			\item The model topoi of presheaves with injective and projective model structures:
			\[
			sPSh(C)_{proj}\to sPSh(C)_{inj}. 
			\]
			\item The model topoi of prestacks:
			\[
			(C,S)^\wedge_{proj}\to (C,S)^\wedge_{inj}.
			\]
			\item The model topoi of sheaves:
			\[
			L_{\L}(C,S)^\wedge_{proj}\to L_{\L}(C,S)^\wedge_{inj}.
			\]
		\end{enumerate}
	\end{lemma} 
	\begin{proof}
		(1) is a standard result about model structures on presheaf categories (see, e.g., Lurie \cite[Lemma A.2.8.3]{lurie2009higher}). All projective cofibrations are objectwise (injective) cofibrations and hence the identity functor is a left Quillen functor. Since it also reflects equivalences it is a Quillen equivalence. 
		\par (2) follows from the following standard result about the functoriality of Bousfield localization (see, e.g., Hirschhorn \cite[Theorem 3.3.20]{hirschhorn2003model}). Let \(M\overset{F}{\underset{U}{\rightleftarrows}} N\) be a Quillen equivalence. Assume that \(S\) is a class of morphisms in $M$ such that the left Bousfield localization \(L_S M\) exists. Then there is an induced Quillen equivalence:
		\[
		L_S M\overset{L_S F}{\underset{L_{\LL F S} U}{\rightleftarrows}} L_{\LL F S} N.
		\]
		Here \(\LL F\) denotes the composition of the cofibrant replacement functor and \(F.\) Since the identity functor preserves the image of the Yoneda embedding and the localization with respect to cofibrantly replaced set of morphisms is Quillen equivalent to the original localization, we get the result.
		\par (3) also follows from the above theorem of Hirschhorn. Indeed, \(\L\) is mapped to \(\L\) by the identity functor and hence the localizations are equivalent. 
	\end{proof}
	\begin{definition}
		A morphism of relative categories \(f\) between relative sites \((C, S)\)~and~\((D, L)\) is called {\it a morphism of relative sites} if it satisfies the {\it homotopical lifting property for covering families}. That is, for any covering family
		\(
		\{\beta_i:V_i\to f(X)\} 
		\) in \(L^{-1}D\)
		there exists a covering family \(\{\alpha_j:U_j\to X\}\) in \(S^{-1}C\) such that each \(f(\alpha_j)\) factors through some~\(\beta_i\) in~\(L^{-1}D.\)
	\end{definition}
	\begin{proposition}\label{pro:quil_adj}
		Let \(f:(C, S)\to (D, L)\) be a morphism of relative sites. The standard push forward--pull back adjunction \((f_!,f^*)\) induces a Quillen adjunction between sheaf categories 
		\[
		f_!:\Sh(C,S)\rightleftarrows \Sh(D,L):f^*.
		\]
	\end{proposition}
	\begin{proof}
		We have such an adjunction at the level of the categories of simplicial presheaves. Since \(f\) is a morphism of relative sites, the Quillen adjunction on presheaf categories preserves collections of morphisms defining sheaf categories as the left Bousfield localizations. Thus functoriality of the left Bousfield localization (see, e.g., Hirschhorn \cite[Theorem 3.3.20]{hirschhorn2003model}) implies that the Quillen adjunction on presheaf categories descends to the Quillen adjunction between sheaf categories. 
	\end{proof}
	We have the following key comparison result for morphisms of relative sites. 
	\begin{theorem}\label{thm:eq_sh_top}
		Let \(f:(C,S)\to (D,L)\) be a morphism of relative sites such that
		\begin{enumerate}
			\item \(f\) is a Dwyer--Kan equivalence (see Dwyer--Kan \cite{dwyer1980simplicial}) of relative categories.
			\item \(f\) reflects homotopical coverages, i.e., \(\{U_i\to X\}\) is a covering family in \(S^{-1}C\) iff \(\Ho(f)(\{U_i\to X\})\) is a covering family in \(L^{-1}D.\)
		\end{enumerate}
		Then \((f_!,f^*)\) is a Quillen equivalence between \(\Sh(C,S)\) and \(\Sh(D,L).\)
	\end{theorem}
	\begin{proof}
		By T\"{o}en--Vezzosi \cite[Corollary 2.3.6]{toen2005homotopical} (original ideas are due to Dwyer--Kan \cite{dwyer1987equivalences}) the first condition implies that there is a Quillen equivalence between the prestack categories 
		\[
		f_!:(C,S)^\wedge\rightleftarrows (D,L)^\wedge:f^*.
		\]
		By the second assumption \(f_!\) reflects the class of hypercovers. Hence we can apply Hirschhorn's theorem on the equivalence of Bousfield localizations \cite[Theorem 3.3.20]{hirschhorn2003model} to left localizations of prestack categories and obtain a Quillen equivalence:
		\[
		f_!:\Sh(C,S)\rightleftarrows \Sh(D,L):f^*.\qedhere
		\]
	\end{proof}
	\begin{definition}\label{def:eq_rel_sites}
		A morphism of relative sites satisfying conditions of Theorem~\ref{thm:eq_sh_top} is referred to as an {\it equivalence of relative sites}. 
	\end{definition}
	\subsection{Geometry over Fermat theories}
	\subsubsection{\(\Tc\)-algebras}\label{sec:TAlg}
	\begin{definition}\label{def:RCom}
		For a commutative ring \(R\) define a Lawvere theory \(\Com_R\) by specifying
		\[
		\Com_R(R^m;R^n)=\{ \text{the set of } R\text{-polynomial maps } \AA^m_R\to \AA^n_R\}.
		\]
		Clearly, this is just the theory of commutative \(R\)-algebras.
	\end{definition}
	\begin{definition}[{\cite[\S 1]{dubuc19841}}]\label{def:fermat}
		Consider a Lawvere theory \(\Tc\) together with a morphism of Lawvere theories
		\[
		\iota:\Com_\Z\to \Tc.
		\]
		It is a {\it Fermat theory} if additionally for arbitrary objects \(S,T\in \Tc\) and any morphism 
		\[
		f:\iota(\AA^1_\Z)\times S\to T
		\]
		there exists \(g:\iota(\AA^1_\Z)\times \iota(\AA^1_\Z)\times S\to T\) such that 
		\[
		f(l_1,s)-f(l_2,s)=(l_1-l_2)\cdot g(l_1,l_2,s).
		\]
		Here \(l_1,l_2:\iota\AA^2_\Z\to\iota\AA^1_\Z\) are the images of two projections and the ring operations are the images of corresponding operations in \(\Com_\Z.\) The latter property is often called \emph{Hadamard's lemma} in reference to the result establishing this property for smooth functions on $\R^n.$
	\end{definition}
	\begin{definition}\label{def:Fermat_over_R}
		We say that \(\Tc\) is a Fermat theory over a ring \(R\) if the morphism \(\iota\) from Definition \ref{def:fermat} factors as 
		\[\begin{tikzcd}
			{\Com_\Z} && \Tc \\
			& {\Com_R}
			\arrow["\iota", from=1-1, to=1-3]
			\arrow["{(\Z\to R)_*}"', from=1-1, to=2-2]
			\arrow["{\ol{\iota}}"', from=2-2, to=1-3]
		\end{tikzcd}\]
		Here the left morphism is induced by the canonical ring homomorphism \[
		\Z\to R,\quad 1_\Z\mapsto 1_R.
		\]
	\end{definition}
	\begin{definition}
		Let \(\Tc\) be a Fermat theory.
		A \(\Tc\)-algebra is a finite product preserving functor $\Tc\to\Set.$ 
	\end{definition}
	\begin{definition}
		A homomorphism of \(\Tc\)-algebras (or simply \(\Tc\)-homomorphism) is a natural transformation of such functors.
	\end{definition}
	
	\begin{construction}[{see \cite[\S 2.2.1]{carchedi2012theories}}]\label{con:completion}
		A morphism of Fermat theories \(\Tc\to \Tc'\) induces an adjoint pair of functors:
		\[\begin{tikzcd}
			{\Tc'\Alg} && \Tc\Alg
			\arrow[""{name=0, anchor=center, inner sep=0}, "{\wh{(-)}}", shift left=3, from=1-1, to=1-3]
			\arrow[""{name=1, anchor=center, inner sep=0}, "{(-)_\#}", shift left=3, from=1-3, to=1-1]
			\arrow["\dashv"{anchor=center, rotate=-90}, draw=none, from=0, to=1]
		\end{tikzcd}\]
		The left adjoint is referred to as {\it \(\Tc\)-completion} and the right one is the forgetful functor.
	\end{construction}
	In particular any commutative algebra over \(K\) admits a completion to an algebra over a Fermat theory \(\Tc\) over \(K.\)
	\begin{notation}
		One can recover a \(\Tc\)-algebra \(A\) from the set \(A(1)\) and the action of \(A(n)\) on this set. Thus in the sequel we will denote the set \(A(1)\) simply by \(A.\) 
	\end{notation}
	There is a forgetful functor from \(\Tc\)-algebras to ordinary algebras. Thus any \(\Tc\)-algebra is an ordinary algebra.
	\begin{proposition}[{\cite[Proposition 1.2]{dubuc19841}}]
		For any ring-theoretic ideal \(I\) in a \(\Tc\)-algebra~\(A\) there is a natural structure of a \(\Tc\)-algebra on \(A/I\) that makes \(A\to A/I\) a \(\Tc\)-homomorphism.
	\end{proposition}
	\begin{definition}\label{def:fg_algebra}
		A \(\Tc\)-algebra is {\it finitely generated} if it is a quotient of the algebra \(y\AA^n_\Tc\) for some \(n\ge 0.\)
	\end{definition}
	\begin{definition}
		For a cardinal \(\lambda\) a \(\Tc\)-algebra is called {\it \(\lambda\)-generated} if it is representable as a filtered colimit over \(\mc{I}\) of finitely generated \(\Tc\)-algebras such that \(|\mc{I}|\le \lambda.\)
	\end{definition}
	\begin{definition}\label{def:res_TAlg}
		Denote by \(\Tc\Alg_r\) the category of \(\max(\aleph_1,\sup_n(|\Tc(n,1)|)\)-generated \(\Tc\)-algebras. We will also assume that there is a coverage (see Definition \ref{def:coverage}) on the category \(\Tc\Alg^{\op}.\) This coverage will be implicit throughout, since all constructions do not depend on a particular choice of such a coverage. For examples of coverages on the opposite category of algebras over Fermat theory see \S\ref{sec:Fermat_examples}.
	\end{definition}
	\begin{remark}
		The peculiar choice of the category \(\Tc\Alg_r\) is justified by our construction of the left adjoint functor \(\Qc\) in the Dold--Kan correspondence (see \S\ref{sec:DK_Fermat}). Morally, we just want the category \(\Tc\Alg_r\) to be small so that constructions with sheaf topoi work as expected.  
	\end{remark}
	\begin{lemma}\label{lem:free_for_fermat}
		The forgetful functor \(\UU:\Tc\Alg\to \Mod_R\) for a Fermat theory \(\Tc\) over \(R\)
		admits a left adjoint {\it free} functor \(\FF.\) Moreover, for a free \(R\)-module \(V\) we have:
		\[
		\FF(V)=\colim_{\substack{\text{free } W\sse V,\\ \rank W<\aleph_0}}\left( \FF(W)=y\AA^{\rank W}_\Tc\right).
		\]
	\end{lemma}
	\begin{proof}
		The result follows from the fact that this statement is true for commutative algebras and the completion functor (Construction \ref{con:completion}) is left adjoint, hence it preserves colimits.
	\end{proof}
	\subsubsection{Examples}\label{sec:Fermat_examples}
	There is a vast array of different examples and applications of our result. We list only the three most fundamental ones, other examples can be found in Carchedi--Roytenberg \cite[\S 2.2.3]{carchedi2012theories}. The key reason why Fermat theories are so useful is that they allow a uniform construction of complicated geometric objects (e.g., higher sheaves) from "affine models", for more detail see Carchedi--Roytenberg \cite[Introduction]{carchedi2012theories}. 
	\begin{example}\label{ex:ag}
		The simplest example of a Fermat theory that we already encountered is that of \(\Com_R\) (see Definition \ref{def:RCom}). This is the theory of commutative \(R\)-algebras for a commutative ring \(R.\) In particular, the theory \(\Com_\Z\) plays the r\^{o}le of the initial Fermat theory. Geometry over \(\Com_\Z\) reproduces the standard scheme-theoretic algebraic geometry.
	\end{example}
	\begin{example}\label{ex:cinf}
		Another very important example is provided by the Fermat theory \(\Cc^\infty\) defined via the following formula.
		\[
		\Cc^\infty(m,n)=\Cc^\infty(\R^m,\R^n).
		\]
		The geometry over this theory reproduces smooth (differential) geometry.
	\end{example}
	\begin{example}\label{ex:hol}
		The final theory we discuss is the theory of holomorphic functions \(\O\) defined via the following formula.
		\[
		\O(m,n)=\O(\C^m,\C^n).
		\]
		The geometry over this theory reproduces complex geometry.
	\end{example}
	Now we give two ways to define a coverage on a small category of algebras over a Fermat theory.
	\begin{construction}
		Define the set \(\Specm A\) for a \(\Tc\)-algebra \(A\) as
		\[
		\Specm A=\Nat(A,\AA^1_\Tc).
		\]
		The set \(\Specm A\) (called the {\it maximal spectrum} of \(A\)) can be topologized in a natural way using the {\it Zariski topology}, i.e., we declare the generating collections of open sets for this topology to be
		\[
		U_a=\{p:A\to \AA^1_\Tc\mid p(a)\neq 0\},\quad \text{for } a\in A(1).
		\]
	\end{construction}
	\begin{construction}\label{con:specm_topology}
		We now define a site structure on \(\Tc\Alg^{\op}_r.\)
		A set of morphisms \(\{f_i:B_i\to A\vert i \in I\}\) is a {\it covering family} for a \(\Tc\)-algebra \(A\) if it is dual to an open cover in Zariski topology:
		\[
		\{\Specm(B_i)\to \Specm(A)\mid i \in I\}.
		\]
	\end{construction}
	\begin{remark}
		Construction \ref{con:specm_topology} is relevant for Examples \ref{ex:cinf}, \ref{ex:hol}, and less so for Example~\ref{ex:ag}. 
	\end{remark}
	Another construction is more convenient for algebro-geometric purposes.
	\begin{definition}
		A ring-theoretic ideal \(I\triangleleft A\) in an algebra \(A\) over a Fermat theory \(\Tc\) is called {\it radical} if \(I\) coincides with its radical defined as
		\[
		\sqrt{I}:=\{a\in A\mid \exists n\in \Z_{\ge 0},\; a^n\in I\}.
		\]
	\end{definition}
	There is a Galois-type correspondence between radical ideals in the algebra \(A\) and {\it reduced} quotients of \(A\) (i.e., quotient \(\Tc\)-algebras of \(A\) without nilpotent elements).
	\begin{construction}
		Denote by \(\Zar(A)\) the poset of radical ideals in the algebra \(A.\)
	\end{construction}
	\begin{lemma}[{Tierney \cite[Proposition 3.1]{tierney1976spectrum}}]
		A poset \(\Zar(A)\) of radical ideals for an algebra \(A\) over a Fermat theory \(\Tc\) is a frame in the sense of MacLane--Moerdijk \mbox{\cite[\S IX.1]{maclane2012sheaves}}. The construction \(\Zar(A)\) can be promoted to a functor
		\[
		\Spec:=\Zar^{\op}:\Tc\Alg^{\op}\to \Loc.
		\] 
	\end{lemma}
	\begin{construction}\label{con:spec_topology}
		We now define another site structure on \(\Tc\Alg^{\op}_r.\)
		A set of morphisms \(\{f_i:B_i\to A\vert i \in I\}\) is a {\it covering family} for a \(\Tc\)-algebra \(A\) if it is dual to an open cover of locales (in the sense of MacLane--Moerdijk \cite[\S IX]{maclane2012sheaves}):
		\[
		\{\Spec(B_i)\to \Spec(A)\mid i \in I\}.
		\]
	\end{construction}
	\begin{remark}
		Construction \ref{con:spec_topology} is well-adapted for the purposes of algebraic geometry (Example \ref{ex:ag}). It is, however, not very convenient for Examples \ref{ex:cinf}, \ref{ex:hol} since locales produced by the functor \(\Spec\) are always compact and thus do not model smooth or holomorphic affine spaces adequately. 
	\end{remark}
	\section{The Dold--Kan correspondence}
	\subsection{Connectivity theorem for Fermat theories}\label{sec:conn_thm}
	Assume that \(\Tc\) is a Fermat theory over some ring \(R\) (see Definition \ref{def:Fermat_over_R}).
	In this section we extend Curtis' connectivity theorem (see \cite[Theorem 3.7]{quillen1969rational}) to semifree simplicial algebras over \(\Tc.\) To this end, we utilize connectivity results developed by Quillen (\cite[Theorem 6.12]{quillen1970co}, \cite[Theorem 8.8]{quillen1968homology}) and Andr\`{e} (\cite[Proposition XIII.3]{andre2013homologie}). 
	\par Denote by \(T_0\) the \(\Tc\)-algebra of a point (i.e., the image under the Yoneda embedding of the object in \(\Tc\) indexed by \(0\)). 
	\begin{definition}
		Denote by \(\Tc\Alg\) the category of {\it all} \(\Tc\)-algebras and by \(\s\Tc\Alg\) the category of {\it all} simplicial \(\Tc\)-algebras, i.e., simplicial objects in \(\Tc\Alg.\)
	\end{definition}
	\begin{proposition}\label{prop:model_struct_on_sFermat}
		With the notation as above, there is a cofibrantly generated model structure on \(\s\Tc\Alg.\)
		Denote by \(\mc{S}^{k-1}\) denote the simplicial set \(\Delta^{k-1}/\pd\Delta^{k-1}\) and by \(\mc{D}^k\) denote the simplicial set \(\Delta^k/\Lambda^k_0.\) We have an inclusion of simplicial sets induced by the inclusion of the last face into the \(k\)-simplex: 
		\(
		\mc{S}^{k-1}\hookrightarrow\mc{D}^k.
		\)
		Denote by \(\ol{R}[-]\) the reduced simplicial chains functor \(\sSet_*\to \s\Mod_R.\) There is a model structure on \(\s\Tc\Alg\) with weak equivalences created by the forgetful functor \(\UU:\s\Tc\Alg\to \sSet_*\) and generating cofibrations given by inclusions:
		\[
		\FF(\ol{R}[\mc{S}^{k-1}])\hookrightarrow \FF(\ol{R}[\mc{D}^{k-1}]).
		\]
		Here \(\FF\) denotes the free functor of \(\Tc\) applied degree-wise.
	\end{proposition}
	\begin{proof}
		For DGAs this is the standard result that can be found in, e.g., Gelfand--Manin \cite[\S V.3]{gelfand2013methods}. For simplicial algebras this is also standard and follows from, e.g., Quillen~\mbox{\cite[Remarks II.4.2]{quillen2006homotopical}}.
	\end{proof}
	\begin{definition}\label{def:cell}
		Let \(\Mc\) be a cofibrantly generated model category. A {\it cellular} object in \(\Mc\) is an object that can be obtained as a transfinite composition of cobase changes for generating cofibrations.
	\end{definition}
	For convenience we denote by \(T_0\) the constant simplicial algebra corresponding to a \(\Tc\)-algebra \(T_0.\) Observe that each algebra \(\A\) in \(\s\Tc\Alg\) which is cellular with respect to the cofibrantly generated model structure on \(\s\Tc\Alg\) constructed in Proposition \ref{prop:model_struct_on_sFermat} is augmented over \(T_0.\) This augmentation is canonical with respect to the cellular structure. The construction of such augmentation is inductive: assume that we have an augmented simplicial cellular \(\Tc\)-algebra \(\B.\) Suppose a cellular algebra \(\A\) is defined via the following cobase change: 
	\[\begin{tikzcd}
		{\FF(\mc{S}^{k-1})} && {\FF(\mc{D}^k)} \\
		\\
		\B && \A.
		\arrow[hook, from=1-1, to=1-3]
		\arrow[from=1-1, to=3-1]
		\arrow[from=3-1, to=3-3]
		\arrow[from=1-3, to=3-3]
		\arrow["\lrcorner"{anchor=center, pos=0.125}, draw=none, from=3-3, to=1-1]
	\end{tikzcd}\]
	Then both \(\FF(\mc{S}^{k-1})\) and \(\FF(\mc{D}^k)\) are augmented in a way that turns the canonical map \(\FF(\mc{S}^{k-1})\to \FF(\mc{D}^k)\) into a map of augmented algebras. This augmentation annihilates all "monomials of positive lengths" in both algebras. Concretely it comes from applying the functor \(\FF\) to the zero morphisms of simplicial vector spaces. By induction, \(\B\) is augmented in a compatible way with the map \(\FF(\mc{S}^{k-1})\to \B.\) Thus we have the following analog of Definition \ref{def:aug_ideal} for arbitrary Fermat theories. 
	\begin{definition}\label{def:aug_ideal_Fermat}
		For a cellular \(\A\) in \(\s\Tc\Alg\) define the {\it augmentation ideal} \(\ol{\A}\) as the kernel of the canonical augmentation map defined above.
	\end{definition}
	\begin{theorem}\label{thm:conn_Fermat}
		Let \(\A\) be a reduced cellular simplicial algebra over \(\Tc.\) Then \(\ol{\A}^r\) is an \mbox{\((r-1)\)-connected} simplicial abelian group. 
	\end{theorem}
	To prove this result we want to invoke \cite[Theorem 6.12]{quillen1970co}. We start by recalling the following definition. 
	\begin{definition}[{Quillen \cite[Definition 6.10]{quillen1970co}}] Let \(A\) be a commutative ring and let \(\Lambda_A\) denote the exterior \(A\)-algebra functor on the category of \(A\)-modules. 
		An ideal \(I\) in \(A\) is called {\it quasiregular} if \(I/I^2\) is a flat \(A/I\)-module and there is a canonical isomorphism
		\begin{equation}\label{eq:quasireg_ideal}
			\Lambda_A I/I^2\xrightarrow{\cong} \Tor^A(A/I,A/I).
		\end{equation}
	\end{definition}
	We have the following key property of free algebras over a Fermat theory:
	\begin{lemma}\label{lem:quasireg_Fermat}
		Let \(A=\FF(V)\) be a free algebra over a Fermat theory \(\Tc\) induced by a free \(T_0\)-module \(V.\) Then the kernel \(\ol{A}\) of the augmentation map \[\ve=\FF(V\to 0)=A\to T_0\] is a quasiregular ideal. 
	\end{lemma}
	\begin{proof}
		The first condition is evident from the Fermat property (Definition \ref{def:Fermat_over_R}). Indeed, \(\ol{A}/\ol{A}^2\) is canonically isomorphic to \(V\) by a Taylor series argument and hence it is a flat module over \(T_0.\)
		\par Let \(V=T_0\la e_i\ra_{i\in I}.\) We denote by \(\{e_i^0\}_{i\in I}\) the generators of the algebra \(A\) corresponding to the aforementioned base of \(V.\) To show the isomorphism \eqref{eq:quasireg_ideal}, consider the following classical  (Koszul) free resolution of the \(A\)-module \(T_0:\)
		\[K_\bullet(A,T_0)=\begin{tikzcd}
			\ldots & {\Lambda^2_A A\la e^2_i\ra_{i\in I}} & {\Lambda^1_A A\la e^1_i\ra_{i\in I}} & A & T_0.
			\arrow["\ve", from=1-4, to=1-5]
			\arrow["{d_0}", from=1-3, to=1-4]
			\arrow["{d_1}", from=1-2, to=1-3]
			\arrow[from=1-1, to=1-2]
		\end{tikzcd}\]
		Here \(d_j\) acts by the Koszul rule:
		\[
		d_j(e^{j+1}_{i_0}\wedge \ldots \wedge e^{j+1}_{i_j})=\sum_{s=0}^j (-1)^se^0_{i_s} e_{i_0}^{j}\wedge\ldots\wedge \wh{e}_{i_s}^j\wedge\ldots\wedge e_{i_j}^j. 
		\]
		Taking the tensor product \(K_\bullet(A,T_0)\otimes T_0\) clearly annihilates the differentials and thus we have the isomorphism
		\[
		\Lambda_A\ol{A}/\ol{A}^2\cong \Lambda_A V\xrightarrow{\cong} K_\bullet(A,T_0)\otimes T_0\cong (\Lambda_A V,d=0)=\Tor^A(T_0,T_0).\qedhere
		\]
	\end{proof}
	\begin{proof}[Proof of Theorem \ref{thm:conn_Fermat}]
		We start by observing that if \(\A\) is a cellular simplicial algebra over \(\Tc\) then it is free in each degree. Moreover, the degree-wise restriction of the canonical augmentation \(\A\to T_0\) constructed in Definition \ref{def:aug_ideal_Fermat} coincides with the canonical augmentation of free \(\Tc\)-algebras.
		\par Consider the augmentation ideal \(\ol{\A}\) of the algebra \(\A.\) By Lemma \ref{lem:quasireg_Fermat}, \(\ol{\A}_k\) is quasiregular for each \(k.\) Additionally, \(\ol{\A}_0=0\) by the assumption that \(\A\) is reduced. Thus we can apply the result of Quillen \cite[Theorem 6.12]{quillen1970co} (for a detailed proof see Quillen \cite[Theorem 8.8]{quillen1968homology})  which implies that \(\ol{\A}^r\) is~\((r-1)\)-connected. 
	\end{proof}
	\subsection{The Dold--Kan correspondence for commutative algebras in characteristic zero \`{a} la Quillen}\label{sec:DK_com}
	This section is an expansion and slight generalization of the Remark on p.223 of Quillen's \enquote{Rational Homotopy Theory} \cite{quillen1969rational}. This remark asserts a Quillen equivalence between categories of connected CDGA over a field of characteristic \(0\) and connected simplicial commutative algebras over the same field. We explicitly construct such an equivalence and extend it to the non-connected case.
	\par Throughout this section we fix a field \(K\) of characteristic \(0.\) We denote by \(\sCAlg\) the category of commutative simplicial algebras over \(K\) and by \(\DGCAlg\) the category of graded-commutative differential graded algebras over \(K.\)
	\begin{proposition}\label{prop:monoidal_moore}
		The functor \(N:\s\Vect_K\to \Ch_K\) of normalized chains induces a functor
		\[
		\N:\sCAlg\to\DGCAlg.
		\]
	\end{proposition}
	\begin{proof}
		Given a commutative simplicial algebra \[(\A,\mu:\A^{\otimes2}\to \A,1_\A:K\to\A),\] we have:
		\[
		\N(\A)=N(\UU\A).
		\]
		Here \(\UU:\sCAlg\to \s\Vect_K\) is the forgetful functor. 
		The multiplication \(\wh{\mu}\) on \(\N(\A)\) is defined as follows:
		\[\begin{tikzcd}
			{\N_p(\A)\otimes\N_q(\A)} && {\N_{p+q}(\A\otimes \A)} && {\N_{p+q}(\A).}
			\arrow["\EZ", from=1-1, to=1-3]
			\arrow["{\N_{p+q}(\mu)}", from=1-3, to=1-5]
			\arrow["{\wh{\mu}}"', curve={height=24pt}, from=1-1, to=1-5]
		\end{tikzcd}\]
		Here \(\EZ\) is the Eilenberg--Zilber map (see Quillen \cite[\S I.4]{quillen1969rational}). The associativity and graded-commutativity properties of the Eilenberg--Zilber map yield associativity and graded commutativity of \(\wh{\mu}.\) The unit of \(\N(\A)\) is defined as the image of \(1_\A\) in \({\N(\A)_0=\A_0.}\)
	\end{proof}
	\begin{proposition}\label{prop:standard_model_struct}
		We have the following two standard cofibrantly generated model structures:
		\begin{enumerate}[(i))]
			\item Let \(S^{k-1}\) be a chain complex that has the ground field in degree \(k-1\) with the zero differential. Let \(D^k\) be a chain complex with two copies of the ground field in degrees \(k-1,k\) and the identity differential between them. \(S^{k-1}\hookrightarrow D^k\) is the obvious inclusion. By convention we put \(S^{-1}=0\) and \(D^0=K[0]\) (a complex concentrated in degree~\(0\)).
			There is a model structure on \(\DGCAlg\) with weak equivalences given by quasi-isomorphism and a set of generating cofibrations given by the inclusions:
			\[
			\SS^g(S^{k-1})\hookrightarrow \SS^g(D^k),\quad k\ge 0.
			\]
			Here \(\SS^g\) is the graded symmetric algebra functor \(\Ch_K\to \DGCAlg.\)
			\item Let \(\mc{S}^{k-1}\) denote the simplicial set \(\Delta^{k-1}/\pd\Delta^{k-1}\) and let \(\mc{D}^k\) denote the simplicial set \(\Delta^k/\Lambda^k_0.\) Then we have an inclusion of simplicial sets induced by the inclusion of the last face into a \(k\)-simplex: 
			\(
			\mc{S}^{k-1}\hookrightarrow\mc{D}^k.
			\)
			By convention we put \(\mc{S}^{-1}=\emptyset.\)
			Denote by \(\ol{K}[-]\) the reduced simplicial chains functor \(\sSet\to \s\Vect_K.\) Then there is a model structure on \(\sCAlg\) with weak equivalences created by the forgetful functor \(\UU:\sCAlg\to \sSet\) and generating cofibrations given by the inclusions:
			\[
			\SS(\ol{K}[\mc{S}^{k-1}])\hookrightarrow \SS(\ol{K}[\mc{D}^{k}]),\quad k\ge 0.
			\]
			Here \(\SS\) denotes the symmetric algebra functor applied degreewise.
		\end{enumerate}
	\end{proposition}
	\begin{proof}
		This follows from \cite[Remarks II.4.2]{quillen2006homotopical}.
	\end{proof}
	\begin{proposition}\label{prop:existence_and_freeness}
		The functor \(\N\) has a left adjoint \(\Qc.\) Consider the standard cofibrantly generated model structures of Proposition \ref{prop:standard_model_struct} on \(\DGCAlg\) and \(\sCAlg.\) The functor~\(\Qc\)~carries cellular objects in \(\DGCAlg\) into cellular objects in \(\sCAlg.\)
	\end{proposition}
	\begin{proof}
		Suppose \(\A\) is a commutative DGA. Denote by \(\UU:\DGCAlg\to \Ch_K\) the forgetful functor. Let \(\Gamma\UU\A\) be the simplicial vector space corresponding to \(\UU\A\) under the (non-monoidal) Dold--Kan correspondence. Denote by \(\SS\Gamma\UU\A\) the symmetric algebra functor applied to \(\Gamma\UU\A\) degreewise. 
		\par If \(x\in \A_p,\) then \(\Gamma x\in \Gamma\UU\A\sse\SS\Gamma\UU\A\) is the element corresponding to \(x\) under the identification \(\Gamma\UU\A\supseteq N\Gamma\UU\A\cong\UU\A.\) It is now clear that for any commutative simplicial algebra \(\B\) there is a bijection between morphisms
		\(
		\phi:\A\to\N\B
		\) in \(\Ch_K\)
		and morphisms \(\theta:\SS\Gamma\UU\A\to \B\) in \(\sCAlg\) such that 
		\[
		\theta(\Gamma x)=\phi(x),\quad \text{for any } x\in \A.
		\]
		Define 
		\begin{equation}\label{eq:functor_Q_over_com}
			\Qc\A=\SS\Gamma\UU\A/\II(\A).
		\end{equation}
		Here \(\II(\A)\) is the simplicial ideal generated by \((\wh{\mu}(\Gamma x,\Gamma y)-\Gamma \mu(x,y)\mid x,y\in \A).\) Then \(\theta\) induces a map \(\Qc\A\to \B\) if and only if \(\phi\) is a morphism in \(\DGCAlg.\) Thus we established a bijection:
		\[
		\Hom_{\DGCAlg}(\A,\N\B)\xrightarrow{\cong}\Hom_{\sCAlg}(\Qc\A,\B).
		\]
		Observe that the adjunction morphism \(\beta:1_{\DGCAlg}\Rightarrow \N\Qc\) is given by 
		\[
		\A\xrightarrow{\beta} \N\Qc\A,\quad x\xmapsto{\beta}\Gamma x+\II(\A).
		\]
		\par Let \(\A\) be a cellular algebra. Since \(\Qc\) is a left adjoint it is sufficient to show that the following cobase change in \(\DGCAlg\) is mapped into a composition of cobase changes of generating cofibrations in \(\sCAlg:\)
		\[\begin{tikzcd}
			{\SS^g(S^{k-1})} && {\SS^g(D^{k})} \\
			\\
			{\A'} && \A.
			\arrow[from=1-1, to=1-3]
			\arrow[from=1-1, to=3-1]
			\arrow[from=1-3, to=3-3]
			\arrow[from=3-1, to=3-3]
			\arrow["\lrcorner"{anchor=center, pos=0.125}, draw=none, from=3-3, to=1-1]
		\end{tikzcd}\]
		Observe that the map
		\[
		\Qc\SS^g(S^{k-1})\to \Qc\SS^g (D^k)
		\]
		is isomorphic to the map
		\[
		\SS \ol{K}[\mc{S}^{k-1}]\to \SS \ol{K}[\mc{D}^k]
		\]
		which is a generating cofibration and thus the map
		\[
		\Qc\A'\to \Qc\A
		\]
		is a cobase change of a generating cofibration in \(\sCAlg.\)
	\end{proof}
	\begin{proposition}\label{prop:homology_commutes}
		Let \(V:\Ch_K\) be a chain complex over \(K.\) Then the following maps are isomorphisms:
		\[
		\SS^g(HV)\xrightarrow{a}H(\SS^gV)\xrightarrow{b} \pi(\SS\Gamma V).
		\]
	\end{proposition}
	\begin{proof}
		See Quillen \cite[Appendix B, Proposition 2.1]{quillen1969rational}.
	\end{proof}
	Observe that each cellular algebra \(\A\) in \(\DGCAlg\) is canonically augmented with respect to the cellular structure. Indeed, both \(\SS^g(S^{k-1})\) and \(\SS^g(D^k)\) are augmented in a way that turns the canonical map \(\SS^g(S^{k-1})\to \SS^g(D^k)\) into a map of augmented algebras. This augmentation annihilates all monomials of positive lengths in both algebras. We can also assume that \(\A'\) is augmented in a way compatible with the map \(\SS^g(S^{k-1})\to \A'\) by induction. Thus \(\A\) is canonically augmented as a pushout. Similarly, a cellular algebra in \(\sCAlg\) is canonically augmented with respect to the cellular structure.
	\begin{definition}\label{def:aug_ideal}
		For a cellular \(\A\) in \(\DGCAlg\) define the {\it augmentation ideal} \(\ol{\A}\) as the kernel of the canonical augmentation map defined above. Similarly, for cellular \(\A\) in \(\sCAlg\) define the {\it augmentation ideal} \(\ol{\A}\) as the kernel of the canonical augmentation map defined above.
	\end{definition}
	\begin{theorem}\label{thm:eq_on_free}
		If \(\A\) is cellular in \(\DGCAlg\) then \(\beta: \A\to \N\Qc\A\) is a weak equivalence.
	\end{theorem}
	\begin{proof}
		Let \(\ol{\Qc\A}^r\) be the powers of augmentation ideal of \(\Qc\A\) defined in Definition \ref{def:aug_ideal}. Let \(\ol{\A}\) be the augmentation ideal of a cellular algebra \(\A\) in \(\DGCAlg.\) It follows that~\({\ol{\Qc\A}^p\cdot \ol{\Qc\A}^q\sse \ol{\Qc\A}^{p+q}}\) and thus 
		\[
		\beta \ol{\A}^r\sse \N\ol{\Qc\A}^r.
		\]
		Thus there is an induced map 
		\[
		\gr\beta:\gr\A\to \gr \N\Qc\A=\N\gr\Qc\A \quad\text{ by exactness of }\N.
		\]
		By Proposition \ref{prop:existence_and_freeness} the object \(\Qc\A\) is cellular and thus by cellular induction
		\[
		\gr\Qc\A\cong \SS(\ol{\Qc\A}/\ol{\Qc\A}^2).
		\]
		Analogously,
		\[
		\gr\A\cong \SS^g(\ol{\A}/\ol{\A}^2).
		\]
		Below we show that the map \(\gr_1\beta\) induces an isomorphism:
		\[
		\ol{\A}/\ol{\A}^2\xrightarrow{\cong} \ol{\N\Qc\A}/\ol{\N\Qc\A}^2
		\]
		Since the augmentation ideal is independent of the differential on \(\A\) we can identify \(\A\) with \(\SS^g(V)\) for some graded vector space \(V.\) Then the graded vector space \(\ol{\A}/\ol{\A}^2\) is generated by the image of \(V\) (it is isomorphic to \(V\)). By the definition of \(\Qc\) given in the proof of Proposition \ref{prop:existence_and_freeness} we have 
		\[
		\Qc\A=\SS(\Gamma\UU\A)/\II(\A).
		\]
		Thus \(\Qc\A\) is generated by \(V\) and all its degenerations. After modding out degenerations with \(\N\) we are left again with an algebra generated by \(V\) and the augmentation ideal \(\ol{\N\Qc\A}\) is generated as an ideal by \(V.\) Thus we have:
		\[
		\ol{\A}/\ol{\A}^2\cong V\cong \ol{\N\Qc\A}/\ol{\N\Qc\A}^2.
		\]
		Thus \(\gr\beta\) is of the form:
		\[
		\SS^g(V)\xrightarrow{\gr\beta}\N\SS\Gamma V.
		\]
		The latter map is a weak equivalence by Proposition \ref{prop:homology_commutes}. 
		\par Here we adopt the terminology of Gwilliam--Pavlov \cite[\S 3]{gwilliam2018enhancing}. Let \(\ol{\A}^0=\A\) and~\({\ol{\N\Qc\A}^0=\N\Qc\A}\). We have two cofibrant sequences:
		\begin{gather*}
			\boldsymbol{\ul{\A}}:\quad \ul{\A}(n):=\begin{cases}
				\A,& n\ge 0;\\
				\ol{\A}^{-n},& n<0
			\end{cases}, \quad \ul{\A}(n)\to \ul{\A}(n+1):=\begin{cases}
				\A\xrightarrow{\id_\A}\A, & n\ge 0;\\
				\ol{\A}^{-n}\hookrightarrow \ol{\A}^{-n+1}, & n<0,
			\end{cases}\\
			\text{and}\\
			\boldsymbol{\ul{\N\Qc\A}}:\hspace{218pt} \ul{\N\Qc\A}(n):=\begin{cases}
				\N\Qc\A,& n\ge 0;\\
				\ol{\N\Qc\A}^{-n},& n<0,
			\end{cases}\\ \ul{\N\Qc\A}(n)\to \ul{\N\Qc\A}(n+1):=\begin{cases}
				\N\Qc\A\xrightarrow{\id_{\N\Qc\A}}\N\Qc\A, & n\ge 0;\\
				\ol{\N\Qc\A}^{-n}\hookrightarrow \ol{\N\Qc\A}^{-n+1}, & n<0.
			\end{cases}
		\end{gather*}
		If \(\A\) is connected (i.e., \(\A_0\cong K\)) for a large enough \(r\) we have \((\ol{\A}^r)_q=0\) and \(\pi_q(\ol{\Qc\A}^r)=0\) by Theorem \ref{thm:conn_Fermat} for \(\Tc=\Com_K\). Thus the sequences \(\ul{\A}\) and \(\ul{\N\Qc\A}\) are both complete. Hence \cite[Lemma 3.25]{gwilliam2018enhancing} is applicable and we see that since the map \(\gr\beta\) is a weak equivalence (i.e., \(\beta\) induces a \emph{graded equivalence}) we have a weak equivalence 
		\[
		\beta:\A=\ul{\A}(\infty)\to \ul{\N\Qc\A}(\infty)=\N\Qc\A.
		\]
		\par If \(\A\) is not connected consider the augmentation ideal in degree \(0\) denoted by \(\ol{\A}_0.\) We have a short exact sequence:
		\[
		0 \to \ol{\A}_0\to \A\to \A/\ol{\A}_0\to 0.
		\]
		The algebra \(\A/\ol{\A}_0\) is connected and thus the map
		\[
		\A/\ol{\A}_0\to \N\Qc(\A/\ol{\A}_0)
		\]
		is a weak equivalence. It remains to observe that \(\ol{\N\Qc(\A)}_0\) is isomorphic to \(\ol{\A}_0\) and \(\ol{\Qc\A}_0\) and thus we have a map of short exact sequences
		\[\begin{tikzcd}
			0 & {\ol{\A}_0} & \A & {\A/\ol{\A}_0} & 0 \\
			0 & {\ol{\N\Qc\A}_0} & \N\Qc\A & {\N\Qc(\A/\ol{\A}_0)=\N\Qc\A/\ol{\N\Qc\A}_0} & 0.
			\arrow[from=1-1, to=1-2]
			\arrow[from=1-2, to=1-3]
			\arrow[from=1-3, to=1-4]
			\arrow[from=1-4, to=1-5]
			\arrow[from=2-1, to=2-2]
			\arrow[from=1-2, to=2-2]
			\arrow[from=2-2, to=2-3]
			\arrow[from=1-3, to=2-3]
			\arrow[from=2-3, to=2-4]
			\arrow[from=1-4, to=2-4]
			\arrow[from=2-4, to=2-5]
		\end{tikzcd}\]
		This map induces an isomorphism in homology and we are done.
	\end{proof}
	\begin{theorem}\label{thm:DK_for_Com}
		The adjoint pair
		\[\begin{tikzcd}
			\sCAlg && \DGCAlg
			\arrow[""{name=0, anchor=center, inner sep=0}, "\N"', shift right=2, from=1-1, to=1-3]
			\arrow[""{name=1, anchor=center, inner sep=0}, "\Qc"', shift right=2, from=1-3, to=1-1]
			\arrow["\dashv"{anchor=center, rotate=-90}, draw=none, from=1, to=0]
		\end{tikzcd}\]
		is a Quillen equivalence.
	\end{theorem}
	\begin{proof}
		Since any cofibrant object in \(\DGCAlg\) is a retract of a cellular object and \(\Qc\) carries cellular objects to cellular objects, Theorem \ref{thm:eq_on_free} implies that the adjunction morphism \(\beta\) is an equivalence on cofibrant objects. Thus the pair \((\Qc\dashv \N)\) is a Quillen equivalence.
	\end{proof}
	\subsection{The Dold--Kan correspondence for Fermat theories over a field of characteristic zero.}\label{sec:DK_Fermat}
	For this section we fix a field \(K\) of characteristic \(0.\)
	\begin{notation}
		Let \(\Tc\) be a Fermat theory over \(K\) (Definition \ref{def:Fermat_over_R}). Let \(\dg\Tc\Alg\) denote the subcategory of \(\DGCAlg\) consisting of objects with a structure of \(\Tc\Alg\) in degree \(0\) and morphisms that respect this structure. Also denote by \(T_0\) the {\it ring of coefficients} for a Fermat theory \(\Tc,\) i.e., the image under the Yoneda embedding of the \(0\)'th object in \(\Tc.\) 
	\end{notation}
	We have the following analog of Proposition \ref{prop:model_struct_on_sFermat} for DG-algebras over \(\Tc:\)
	\begin{proposition}\label{prop:model_struct_on_DGFermat}
		With the notation as above we have a cofibrantly generated model structure on the category \(\dg\Tc\Alg.\)
		\par Consider two types of chain complexes \(S^{k-1}, D^k\) defined in Proposition \ref{prop:standard_model_struct} (i).
		There is a model structure on \(\dg\Tc\Alg\) with weak equivalences given by quasi-isomorphisms and a set of generating cofibrations given by the inclusions
		\[
		\FF^g(S^{k-1})\hookrightarrow \FF^g(D^k),\quad k\ge 0.
		\]
		Here \(\FF^g\) is the free functor of \(\Tc\) in degree \(0\) and the graded symmetric algebra functor over \(T_0\) in positive degrees.
	\end{proposition}
	\begin{proof}
		This is a result of Carchedi--Roytenberg \cite[\S 6]{carchedi2012homological}.
	\end{proof}
	\begin{proposition}
		Let \(\Tc\) be a Fermat theory as above.
		The Moore complex functor
		\[
		N:\s\Vect_K\to \Ch_K
		\]
		induces a functor between categories of algebras
		\[
		\N:\s\Tc\Alg\to \dg\Tc\Alg.
		\]
	\end{proposition}
	\begin{proof}
		In positive degrees, this result follows from Proposition \ref{prop:monoidal_moore}. In degree \(0\) the functor \(N\) induces an isomorphism and the structure of a \(\Tc\)-algebra is transferred in an evident way.
	\end{proof}
	\begin{proposition}
		The functor \(\N\) has a left adjoint \(\Qc.\) Consider cofibrantly generated model structures of Propositions \ref{prop:model_struct_on_sFermat} and \ref{prop:model_struct_on_DGFermat} on \(\dg\Tc\Alg\) and \(\s\Tc\Alg.\) Then the functor~\(\Qc\) carries cellular objects in \(\DGCAlg\) into cellular (and free) objects in \(\sCAlg.\)
	\end{proposition}
	\begin{proof}
		If \(\A\) is a dg-algebra over \(\Tc.\) If \(x\in \A_p,\) then \(\Gamma x\in \Gamma\UU\A\sse \FF\Gamma\UU\A\) is the element corresponding to \(x\) under the canonical isomorphism \(N\Gamma\UU\A\cong\UU\A.\) Clearly for any object \(\B\) in \(\s\Tc\Alg\) there is a bijection between morphisms \(\phi:\A\to \N\B\) in \(Ch_K\) and morphisms \(\theta:\FF\Gamma\UU\A\to \B\) in \(\s\Tc\Alg\) such that
		\[
		\theta(\Gamma x)=\phi(x),\quad \text{for all } x\in\A.
		\]
		Define 
		\begin{equation}\label{eq:functor_Q_over_T}
			\Qc\A=\FF\Gamma\UU\A/\II(\A).
		\end{equation}
		Here \(\II(\A)\) is the simplicial ideal generated defined via the following formula.
		\[\II(\A):=(\wh{f}(\Gamma x_1,\ldots, \Gamma x_n)-\Gamma f(x_1,\ldots,x_n)\mid x_1,\ldots,x_n\in \A)\] for all operations \(f\) of \(\Tc\) that can be transferred along \(\N\) (i.e., all operations in degree \(0\) and multiplications in positive degrees). Observe that the adjunction morphism~\({\beta:1_{\dg\Tc\Alg}\Rightarrow\N\Qc}\) is given by
		\[
		\A\xrightarrow{\beta} \N\Qc\A,\quad x\xmapsto{\beta}\Gamma x+\II(\A).
		\]
		The proof of the fact that cellular objects are mapped to cellular (and free) objects by \(\Qc\) carries over mutatis mutandis from the proof of Proposition \ref{prop:existence_and_freeness}.
	\end{proof}
	\begin{remark}
		Previous results hold without modifications for any (multisorted) Lawvere theory \(\Tc.\) The next result, however, uses the Fermat property in an essential way.
	\end{remark}
	\begin{proposition}\label{prop:homology_commutes_with_completion}
		Let \(V\) be a simplicial vector space over \(K\). Then there is an isomorphism induced by the completion functor (Construction \ref{con:completion}):
		\[
		T_0\otimes_K\pi_{>0}(\SS V)=\pi_{>0}(\SS T\otimes_K V)\xrightarrow{\cong} \pi_{>0}(\FF V).
		\]
		In dimension \(0\) we have \(\pi_0(\FF V)=\FF\pi_0(V).\)
	\end{proposition}
	\begin{proof}
		First, observe that the equality \(T_0\otimes_K\pi_*(\SS V)=\pi_*(\SS T\otimes_K V)\) is just a standard result about the flat base change.
		Now, any simplicial vector space can be obtained as a transfinite composition of the cobase changes
		\[\begin{tikzcd}
			\Gamma{\ol{K}[\mc{S}^{n-1}]} && \Gamma{\ol{K}[\mc{D}^n]} \\
			\\
			{V'} && V.
			\arrow[from=1-1, to=1-3]
			\arrow[from=1-1, to=3-1]
			\arrow[from=3-1, to=3-3]
			\arrow[from=1-3, to=3-3]
			\arrow["\lrcorner"{anchor=center, pos=0.125}, draw=none, from=3-3, to=1-1]
		\end{tikzcd}\]
		The left adjoint functor \(\SS\) preserves colimits and thus \(\SS(V)\) is a transfinite composition of cobase changes of the form:
		\[\begin{tikzcd}
			\SS\Gamma{\ol{K}[\mc{S}^{n-1}]} && \SS\Gamma{\ol{K}[\mc{D}^n]} \\
			\\
			{\SS V'} && \SS V.
			\arrow[from=1-1, to=1-3]
			\arrow[from=1-1, to=3-1]
			\arrow[from=3-1, to=3-3]
			\arrow[from=1-3, to=3-3]
			\arrow["\lrcorner"{anchor=center, pos=0.125}, draw=none, from=3-3, to=1-1]
		\end{tikzcd}\]
		Since the completion functor is also a left adjoint, the algebra \(\FF(V)\) is a composition of cobase changes:
		\[\begin{tikzcd}
			\FF\Gamma{\ol{K}[\mc{S}^{n-1}]} && \FF\Gamma{\ol{K}[\mc{D}^n]} \\
			\\
			{\FF V'} && {\FF V.}
			\arrow[from=1-1, to=1-3]
			\arrow[from=1-1, to=3-1]
			\arrow[from=3-1, to=3-3]
			\arrow[from=1-3, to=3-3]
			\arrow["\lrcorner"{anchor=center, pos=0.125}, draw=none, from=3-3, to=1-1]
		\end{tikzcd}\]
		Thus to show the desired isomorphism we only need to show the result for \(V=\ol{K}[\mc{S}^n],\) since \(\FF\Gamma{\ol{K}[\mc{D}^n]}\) is acyclic. By Proposition \ref{prop:homology_commutes} we have an isomorphism
		\begin{equation}\label{eq:iso_for_sphere}
			\pi(\SS\Gamma\ol{K}[\mc{S}^n])\cong \SS^g(\ol{K}[\mc{S}^n]).
		\end{equation}
		Moreover, in each dimension we can explicitly describe \(\ker d_i\) for \(\SS\Gamma\ol{K}[\mc{S}^n]\) as well as for \(\FF\Gamma\ol{K}[\mc{S}^n].\) Let \(x_n\) be the generator of \(\ol{K}[\mc{S}^n].\) By Hadamard's lemma (see Definition \ref{def:fermat}) the ideal \(\ker d_i^k\) is given by
		\begin{gather}\label{eq:kernel_ideal}
			\begin{split}
				\ker d_i^k=\left(s_{j_{k-n}}\ldots s_{j_t=i}\ldots s_{j_1}x_n-s_{j_{k-n}}\ldots s_{j_t=i-1}\ldots s_{j_1}x_n\mid j_{k-n}>\ldots>j_1,\; j_{t-1}<i-1\right)+\\+\left(s_{j_{k-n}}\ldots s_{j_1}x_n\mid i,i-1\notin \{j_1,\ldots, j_{k-n}\}\right).
			\end{split}
		\end{gather}
		The intersection of such ideals is generated by products of polynomials described in~\eqref{eq:kernel_ideal}. The isomorphism \eqref{eq:iso_for_sphere} guarantees that the quotient
		\begin{equation}\label{eq:quotient}
			\pi_k(\SS\Gamma\ol{K}[\mc{S}^n])=\frac{\bigcap\limits_{i=0}^k \ker d_i^k}{d_0^{k+1}\bigcap\limits_{i=1}^{k+1}\ker d_i^{k+1}}
		\end{equation}
		is a finite-dimensional vector space over \(K.\) In particular, this quotient has a finite basis consisting of polynomials \(\{p_1,\ldots,p_l\}.\) Upon completion, this collection of polynomials still generates the ideal \(\bigcap\limits_{i=0}^k\ker d_i^k\) together with \(d_0^{k+1}\bigcap\limits_{i=1}^{k+1}\ker d_i^{k+1}.\) However, we want to show that it is an additive basis of the quotient, but now over \(T_0\). Let \(f\in \bigcap\limits_{i=0}^k\ker d_i^{k+1}\) be an arbitrary element \(f=\sum_{i=1}^l \lambda_i p_i+\gamma,\) where \(\gamma\in d_0^{k+1}\bigcap\limits_{i=1}^{k+1}\ker d_i^{k+1}.\) Denote by \(N\) the maximal total degree of the collection \(\{p_1,\ldots,p_l\}.\) For a vector of non-negative integers \(s=(s_1,\ldots,s_l)\) we denote by \(p^s\) the corresponding monomial \(\prod_{i=1}^l p_i^{s_i}.\) For such a vector \(s\) we denote by \(|s|\) the sum of its components. Now we take a Taylor expansion of \(\lambda_i\)'s up to the total degree \(N+1.\) We get
		\[
		\sum_{|s|\le N+1} \left[\lambda^{(s)}_i p^s+\xi_i p_i \right] + \gamma.
		\]
		Here \(\xi_i\) is a remainder of the Taylor series of \(\lambda_i.\) In particular, \(\xi_i p_i\) is divisible by a polynomial of total degree \(>N\) so it belongs to \(d_0^{k+1}\bigcap\limits_{i=1}^{k+1}\ker d_i^{k+1}.\) Thus we have shown that \(p_i\)'s span the quotient \eqref{eq:quotient} and they are clearly linearly independent over \(T_0.\) Thus we have an isomorphism:
		\[
		\pi_k(\SS\Gamma\ol{K}[\mc{S}^n])\cong \pi_k(\FF\Gamma\ol{K}[\mc{S}^n]),\quad k>0
		\]
		induced by completion.
		\par To resolve the case \(k=0\) we observe that the statement is trivially true for \(\Gamma\ol{K}[\mc{S}^n]\) with \(n>0\). In the case \(n=0\) we have an identity \(\pi_0(\Gamma\ol{K}[\mc{S}^0])=\Gamma\ol{K}[\mc{S}^0]\) yielding
		\[
		\pi_0(\FF\Gamma\ol{K}[\mc{S}^0])=\FF\pi_0(\Gamma\ol{K}[\mc{S}^0]).
		\]
		Thus the statement holds for all simplicial vector spaces.
	\end{proof}
	\begin{corollary}
		Let \(V\) be a chain complex over \(K.\) Then we have two isomorphisms
		\[
		\FF^g H(V)\xrightarrow{a} H(\FF^g V)\xrightarrow{b}\pi(\FF\Gamma V).
		\]
	\end{corollary}
	\begin{proof}
		Since over \(K\) any chain complex is equivalent to its homology and \(\FF^g\) preserves homotopy equivalences we have a sequence of equivalences
		\[
		H(\FF^g H(V))=\FF^g H(V)\simeq \FF^g V\simeq H(\FF^g V).
		\]
		Thus \(a\) is an isomorphism. To see that \(b\) is an isomorphism apply Proposition \ref{prop:homology_commutes} and Proposition \ref{prop:homology_commutes_with_completion}.
	\end{proof}
	\begin{theorem}
		If \(\A\) is a cellular object in \(\dg\Tc\Alg\) then the adjunction morphism~\({\beta:\A\to \N\Qc\A}\) is an equivalence. 
	\end{theorem}
	\begin{proof}
		The proof is identical to the proof of Theorem \ref{thm:eq_on_free}.
	\end{proof}
	\begin{theorem}\label{thm:DK_for_Fermat}
		The adjoint pair
		\[\begin{tikzcd}
			\s\Tc\Alg && \dg\Tc\Alg
			\arrow[""{name=0, anchor=center, inner sep=0}, "\N"', shift right=2, from=1-1, to=1-3]
			\arrow[""{name=1, anchor=center, inner sep=0}, "\Qc"', shift right=2, from=1-3, to=1-1]
			\arrow["\dashv"{anchor=center, rotate=-90}, draw=none, from=1, to=0]
		\end{tikzcd}\]
		is a Quillen equivalence.
	\end{theorem}
	\begin{proof}
		The proof is identical to the proof of Theorem \ref{thm:DK_for_Com}.
	\end{proof}
	\section{Three models of affine derived stacks}\label{sec:models}
	Throughout this section \(\Tc\) is a Fermat theory over a field \(K\) of characteristic \(0\) (see Definition \ref{def:Fermat_over_R}).
	\subsection{Simplicial \(\Tc\)-algebras}
	\begin{definition}
		Denote by \(\s\Tc\Alg_r\) the category  of simplicial objects in the category~\(\Tc\Alg_r\) (see Definition \ref{def:res_TAlg}) with an additional assumption that the functor \(\pi_*\) yields \emph{finitely generated} \(\Tc\)-algebras for all objects in this category. 
	\end{definition}
	\begin{construction}\label{con:cov_on_sTAlg}
		A collection \(\{f_i:B_i\to A\vert i\in I\}\) of morphisms in \(\s\Tc\Alg_r^{\op}\) is a covering family of a simplicial \(\Tc\)-algebra \(A\) if 
		\begin{itemize}
			\item \(\{\pi_0(f_i): \pi_0(B_i)\to \pi_0(A)\}\) is a covering family for some coverage on \(\Tc\Alg^{\op}_r\) (see \S\ref{sec:Fermat_examples} for examples);
			\item \(\pi_*(B_i)\xleftarrow{\cong}\pi_*(A)\otimes_{\pi_0(A)}\pi_0(B_i).\)
		\end{itemize}
	\end{construction}
	Let \(\s\Vect_K\) denote the category of simplicial vector spaces over \(K.\)
	\begin{proposition}
		There is a left adjoint \(\s\FF\) of the forgetful functor 
		\[
		\UU:\s\Tc\Alg\to \s\Vect_K
		\]
		such that \[(\s\FF K)_n= \FF(K_n).\]
		Here \(\FF\) is the free \(\Tc\)-algebra functor on a vector space (see Lemma \ref{lem:free_for_fermat}).
	\end{proposition}
	\begin{construction}
		The category \(\s\Tc\Alg\) carries a model structure transferred from the category \(\s\Vect_K\) along the functor \(\UU.\) We restrict \(\s\Tc\Alg_r\) to give this category the structure of a small model category. 
	\end{construction}
	\subsection{Commutative differential \(\Tc\)-algebras}
	Another way to construct affine derived \mbox{\(\Tc\)-spaces} is to consider them as duals of differential~\(\Tc\)-algebras.
	\begin{definition}
		Denote by \(\dg\Tc\Alg_r\) the intersection of the category \(\dg\Tc\Alg\) with the category of \(\max(\aleph_1,\sup_n(|\Tc(n,1)|)\)-generated CDGAs over \(K\) with the additional assumption that the homology of all objects in this category is finitely generated (as a \(\Tc\)-algebra).
	\end{definition}
	\begin{construction}\label{con:cov_on_DGTAlg}
		A collection \(\{f_i:B_i\to A\vert i\in I\}\) is a covering family of a DG \(\Tc\)-algebra \(A\) if 
		\begin{itemize}
			\item \(\{H_0(f_i): H_0(B_i)\to H_0(A)\}\) is a covering family for some coverage \(\Tc\Alg^{\op}_r\) (see \S\ref{sec:Fermat_examples} for examples);
			\item \(H_*(B_i)\xleftarrow{\cong}H_*(A)\otimes_{H_0(A)}H_0(B_i).\)
		\end{itemize}
	\end{construction}
	
	\subsection{Semifree differential graded \(\Tc\)-algebras}
	Finally, yet another model is given by duals of semifree commutative differential graded \(\Tc\)-algebras.
	\begin{definition}
		A differential graded algebra is semifree if the underlying graded algebra is free.
	\end{definition}
	\begin{notation}
		We will denote the full subcategory of \(\dg\Tc\Alg_r\) spanned by semifree differential graded \(\Tc\)-algebras by \(\sfdg\Tc\Alg_r.\)
	\end{notation}
	\begin{construction}
		A homotopical coverage on \(\sfdg\Tc\Alg_r^{\op}\) is transferred along the inclusion functor \[\sfdg\Tc\Alg_r\hookrightarrow \dg\Tc\Alg_r.\]
	\end{construction}
	\section{Main results}\label{sec:main}
	\begin{theorem}\label{thm:eq_sf}
		The inclusion functor \(\iota\) defines a Dwyer--Kan equivalence between relative categories:
		\[\begin{tikzcd}
			\sfdg\Tc\Alg_r && \dg\Tc\Alg_r
			\arrow["\iota", from=1-1, to=1-3]
			\arrow["\simeq"', from=1-1, to=1-3]
		\end{tikzcd}\]
	\end{theorem}
	\begin{proof}
		This is the standard result about inclusion of the subcategory of cofibrant objects into a model category. See, for instance, \DHKS \cite[Corollary 10.4]{dwyer2004homotopy}.
	\end{proof}
	\begin{proposition}\label{pro:DK_preserves_coverages}
		The functors \(\Qc,\N\) in the Dold--Kan correspondence preserve the coverages on the categories \(\s\Tc\Alg^{\op}_r\) and \(\dg\Tc\Alg^{\op}_r.\) That is, if \(\{\B_i\to\A\}\) is a covering family in \(\s\Tc\Alg^{\op}_r,\) then \(\{\N\B_i\to \N\A\}\) is a covering family in \(\dg\Tc\Alg^{\op}_r\) and the analogous statement holds for \(\Qc.\)
	\end{proposition}
	\begin{proof}
		Observe that since homotopy groups of objects in \(s\Tc\Alg_r\) are computed at the level of the underlying simplicial sets and all such sets are fibrant we have
		\[
		\pi_*(\A)=H_*(\N \A),\quad \A\in s\Tc\Alg_r.
		\]
		Similarly, for the objects in \(\dg\Tc\Alg_r\) we have
		\[
		H_*(\A)=\pi_*(\Qc\A),\quad \A\in \dg\Tc\Alg_r.
		\]
		Thus the functor \(\N\) maps Condition \ref{con:cov_on_sTAlg} into Condition \ref{con:cov_on_DGTAlg} and hence preserves covering families. Analogously the functor \(\Qc\) preserves covering families. 
	\end{proof}
	\begin{corollary}
		There is a Quillen adjunction between sheaf categories (see Definition \ref{def:sheaves}) induced by the adjunction on sites
		\[\begin{tikzcd}
			{\Sh(\s\Tc\Alg^{\op}_r)} && {\Sh(\dg\Tc\Alg^{\op}_r)}
			\arrow["{Q^*}"', curve={height=12pt}, from=1-1, to=1-3]
			\arrow["{N^*}"', curve={height=12pt}, from=1-3, to=1-1]
		\end{tikzcd}\]
	\end{corollary}
	\begin{proof}
		By Proposition \ref{pro:DK_preserves_coverages} the functor \(\N^*\) maps hypercovers on \(\dg\Tc\Alg^{\op}_r\) into hypercovers on \(\s\Tc\Alg^{\op}_r.\) Thus by \DHI \cite[Proposition 8.2]{dugger2004hypercovers} the functors \((\N^*,\N_*=\Qc^*)\) form a Quillen adjunction.  
	\end{proof}
	We arrive at our main result about the equivalence of various models for derived stacks:
	\begin{theorem}\label{thm:site_eq}
		There is an equivalence of model sites \[d\Cart_{\Tc}:=\sfdg\Tc\Alg^{\op}_r, \dg\Tc\Alg^{\op}_r, \text{ and } \s\Tc\Alg^{\op}_r.\]
	\end{theorem}
	\begin{proof}
		\par By construction the functor \(\iota^{\op}:d\Cart_{\Tc}\to \dg\Tc\Alg^{\op}_r\) reflects coverages and hence hypercovers. Moreover, by Theorem \ref{thm:eq_sf} the functor \(\iota^{\op}\) is an equivalence of homotopical categories. Thus, \(\iota^{\op}\) is an equivalence of relative sites (Definition~\ref{def:eq_rel_sites}) and we can apply Theorem~\ref{thm:eq_sh_top}. 
		\par Similarly, by Proposition \ref{pro:DK_preserves_coverages} the functor \(\Qc^*\) reflects homotopy coverages and hence hypercovers. Additionally, by Theorem \ref{thm:DK_for_Fermat}, it is a Quillen equivalence. Hence \(\Qc\) is an equivalence of relative sites and we can apply Theorem \ref{thm:eq_sh_top} to see that the sheaf categories over \(\dg\Tc\Alg^{\op}_r\) and \(\s\Tc\Alg^{\op}_r\) are equivalent. 
	\end{proof}
	\section{Applications}\label{sec:apps}
	We state the following slight generalization of Theorem \ref{thm:site_eq}.
	\begin{theorem}\label{thm:site_eq_gen}
		Assume that we have small full subcategories in \({\MM\sse \s\Tc\Alg}\) and \({\NN\sse \dg\Tc\Alg}\). And assume that they are mapped to one another by the pair \((\Qc,\N)\) defined in \S\ref{sec:DK_Fermat}. Assume additionally, that there are homotopical coverages on \(\MM\) and \(\NN\) preserved by the pair \((\Qc,\N).\) Then the pair \((\Qc^*,\N^*)\) induces an equivalence of sheaf topoi
		\[\begin{tikzcd}
			{\Sh(\MM^{\op})} && {\Sh(\NN^{\op}).}
			\arrow["{Q^*}"', curve={height=12pt}, from=1-1, to=1-3]
			\arrow["{N^*}"', curve={height=12pt}, from=1-3, to=1-1]
		\end{tikzcd}\]
	\end{theorem}
	\begin{proof}
		The proof is the same as the proof of the second equivalence in Theorem \ref{thm:site_eq}.
	\end{proof}
	\subsection{Models for derived smooth infinitesimal analysis}\label{subsec:eq_synth}
	In this section, we list several possible sites for derived stacks over the Fermat theory \(\CI\) (see Example \ref{ex:cinf}) following \MR \cite[Appendix 2]{moerdijk2013models}. We then apply Theorem \ref{thm:site_eq_gen} to establish equivalences between sheaf topoi over some of those sites. It is customary to refer to algebras over the theory \(\CI\) as {\it \(\CI\)-rings}; we adhere to this convention. 
	\begin{notation}
		Let \(I\triangleleft \CI(\R^n)\) be an ideal of smooth functions. Denote by \(Z(I)\sse \R^n\) the set of zeroes for this ideal. For a set \(X\sse \R^n\) denote by \(m_{X}^0\) the ideal of smooth functions that vanish on \(X.\)
	\end{notation}
	\begin{definition}\label{def:gm_det}
		An ideal \(I\triangleleft \CI(\R^n)\) is {\it germ determined} if for any \(f\in \CI(\R^n)\setminus I\) there exists \(p\in \R^n\) such that in the localization we have
		\(f_p\notin I_p.\)
	\end{definition}
	\begin{definition}\label{def:Weil_alg}
		A local \(\CI\)-ring \(W\) is {\it a Weil algebra} if it is finite dimensional over \(\R\) and splits as a sum
		\[
		W=\R\la 1\ra\oplus\mf{m}.
		\]
		Here \(\mf{m}\) is the maximal ideal.
	\end{definition}
	\par We start by listing the non-derived versions of the sites we are interested in.
	\begin{construction}\label{con:cinf_sites}
		The category of Loci \(\mb{L}\) is the category opposite to the full subcategory of finitely generated \(\CI\)-rings (see Definition \ref{def:fg_algebra}) in the category of \(\CI\)-rings. The category \(\mb{L}\) is the ambient category for all other sites listed below, which are full subcategories of \(\mb{L}\) specified by a condition on presentations of their objects as quotients \(\CI(\R^n)/I.\)
		\begin{enumerate}
			\item The category \(\mb{E}\) consists of \(\CI(\R^n)/I\) such that \(f\in I\) iff \(f\vert_{Z(I)}\equiv 0.\)
			\item The category \(\mb{F}\) consists of \(\CI(\R^n)/I\) such that \(I\) is closed.
			\item The category \(\mb{G}\) consists of \(\CI(\R^n)/I\) such that \(I\) is germ-determined (see Definition \ref{def:gm_det}).
			\item The category \(\mb{L}_{fp}\) consists of \(\CI(\R^n)/I\) such that \(I\) is finitely generated.
			\item The category \(\mb{V}\) consists of \(\CI(\R^n)/I\) such that \(I\) is \(\CI\)-radical, i.e. \(f\in I\) iff~\(\exists g\in I,\; Z(f)=Z(g).\)
			\item The category \(\mb{V}_\omega\) consists of \(\CI(\R^n)/I\) such that \(I\) is \(\CI\)-radical and countably generated, i.e. there is \(\{g_k\}_{k=0}^\infty\sse \CI(\R^n)\) such that \(f\in I\) iff \(\exists n\; Z(g_n)=Z(f).\)
			\item The category \(\mb{MW}\) consists of \(\CI(\R^n)/I\) isomorphic to \(\CI(M)\otimes_{\infty} W\) for a manifold \(M\) and \(W\cong \CI(\R^{n-k})/J\) is a Weil algebra (see Definition \ref{def:Weil_alg}). Here~\(\otimes_{\infty}\) denotes the coproduct in the category of \(\CI\)-rings.
			\item The category \(\mb{W}\) consists of \(\CI(\R^n)/I\) isomorphic \(A\otimes_\infty W,\) where \(A\in \mb{V}\) and~\(W\) is a Weil algebra.
			\item The category \(\mb{W}_\omega\) consists of \(\CI(\R^n)/I\) isomorphic \(A\otimes_\infty W,\) where \(A\in \mb{V}_\omega\) and~\(W\) is a Weil algebra.
		\end{enumerate}
	\end{construction}
	All the categories listed above are defined to be closed under isomorphism. We have the following diagram of inclusion functors (see \MR \cite[Discussion on p.364]{moerdijk2013models}). Single arrows mean that there is no right adjoint; all other inclusions are coreflective:
	
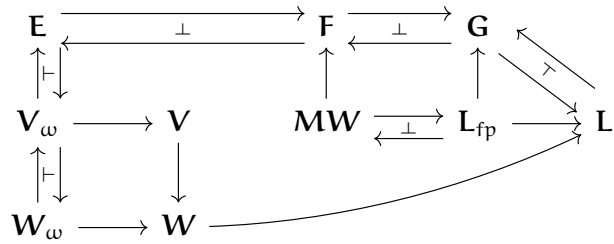
\[\begin{tikzcd}
		{\mb{E}} && {\mb{F}} & {\mb{G}} \\
		{\mb{V}_\omega} & {\mb{V}} & {\mb{MW}} & {\mb{L}_{fp}} & {\mb{L}} \\
		{\mb{W}_\omega} & {\mb{W}}
		\arrow[""{name=0, anchor=center, inner sep=0}, shift left=2, from=1-1, to=1-3]
		\arrow[""{name=1, anchor=center, inner sep=0}, shift left=3, from=1-1, to=2-1]
		\arrow[from=2-1, to=2-2]
		\arrow[""{name=2, anchor=center, inner sep=0}, shift left=2, from=1-3, to=1-1]
		\arrow[""{name=3, anchor=center, inner sep=0}, from=2-1, to=1-1]
		\arrow[""{name=4, anchor=center, inner sep=0}, shift left=3, from=2-1, to=3-1]
		\arrow[from=3-1, to=3-2]
		\arrow[from=2-2, to=3-2]
		\arrow[""{name=5, anchor=center, inner sep=0}, from=3-1, to=2-1]
		\arrow[from=2-3, to=1-3]
		\arrow[""{name=6, anchor=center, inner sep=0}, shift left=1, from=2-3, to=2-4]
		\arrow[""{name=7, anchor=center, inner sep=0}, shift left=2, from=1-3, to=1-4]
		\arrow[from=2-4, to=1-4]
		\arrow[""{name=8, anchor=center, inner sep=0}, shift left=2, from=2-4, to=2-3]
		\arrow[from=2-4, to=2-5]
		\arrow[""{name=9, anchor=center, inner sep=0}, shift right=1, from=1-4, to=2-5]
		\arrow[""{name=10, anchor=center, inner sep=0}, shift right=3, from=2-5, to=1-4]
		\arrow[curve={height=12pt}, from=3-2, to=2-5]
		\arrow[""{name=11, anchor=center, inner sep=0}, shift left=2, from=1-4, to=1-3]
		\arrow["\dashv"{anchor=center, rotate=-90}, draw=none, from=0, to=2]
		\arrow["\dashv"{anchor=center, rotate=-180}, draw=none, from=1, to=3]
		\arrow["\dashv"{anchor=center, rotate=-180}, draw=none, from=4, to=5]
		\arrow["\dashv"{anchor=center, rotate=51}, draw=none, from=9, to=10]
		\arrow["\dashv"{anchor=center, rotate=-90}, draw=none, from=6, to=8]
		\arrow["\dashv"{anchor=center, rotate=-90}, draw=none, from=7, to=11]
	\end{tikzcd}\]
	\begin{definition}
		{\it An open cover} on \(\CI(\R^n)/I\in \LL\) is a family of morphisms
		\[
		\{\CI(\R^n)/I\to \CI(U_i)/(I\vert_{U_i})\mid \{U_i\to \R^n\} \text{ is an open cover}\}
		\]
		 {\it A projection} is a map opposite to the canonical map \(A\to A\otimes_\infty B\) for \(B\neq 0.\) 
	\end{definition}
	The categories of Construction \ref{con:cinf_sites} can be equipped with one of the following four coverages:
	\begin{construction}\label{con:cinf_top}
		\begin{enumerate}[(i)]
			\item\label{iso} Covering families consist of isomorphisms.
			\item\label{fo} Covering families consist of finite open covers.
			\item\label{fop} Covering families consist of finite open covers and projections.
			\item\label{oc} Covering families consist of open covers.
		\end{enumerate}
	\end{construction}
	We have the following diagram of corresponding topoi, where the empty entries correspond to non-subcanonical topologies.
	\[\tag{$\star$}\label{table_of_topoi}\begin{tikzcd}
		& {\ref{oc}} & {\ref{fop}} & {\ref{fo}} & {\ref{iso}} \\
		{\mb{L}} && {\mc{B}} & {\mc{Z}} & {\Set^{\mb{L}^{\op}}} \\
		{\mb{F}} & {\mc{F}} & {\mc{F}_{fin}} & {\mc{F}_{fin}} & {\Set^{\mb{F}^{\op}}} \\
		{\mb{G}} & {\mc{G}} & {\mc{G}_{fin}} & {\mc{G}_{fin}} & {\Set^{\mb{G}^{\op}}} \\
		{\mb{MW}} & {\mc{MW}} & {\mc{MW}} & {\mc{MW}_{fin}} & {\Set^{\mb{MW}^{\op}}} \\
		{\mb{V}} && {\mc{V}_{fpr}} & {\mc{V}_{fin}} & {\Set^{\mb{V}^{\op}}} \\
		{\mb{W}} && {\mc{W}_{fpr}} & {\mc{W}_{fin}} & {\Set^{\mb{W}^{\op}}} \\
		{\mb{E}} & {\mc{E}} & {\mc{E}} & {\mc{E}_{fin}} & {\Set^{\mb{E}^{\op}}}
		\arrow[hook, from=2-3, to=2-4]
		\arrow["\simeq", from=3-3, to=3-4]
		\arrow[hook, from=3-2, to=3-3]
		\arrow[hook, from=4-2, to=4-3]
		\arrow["\simeq", from=4-3, to=4-4]
		\arrow["\simeq", from=5-2, to=5-3]
		\arrow[hook, from=5-3, to=5-4]
		\arrow[hook, from=6-3, to=6-4]
		\arrow[hook, from=7-3, to=7-4]
		\arrow[hook, from=8-3, to=8-4]
		\arrow["\simeq", from=8-2, to=8-3]
		\arrow[hook, from=2-4, to=2-5]
		\arrow[hook, from=3-4, to=3-5]
		\arrow[hook, from=4-4, to=4-5]
		\arrow[hook, from=5-4, to=5-5]
		\arrow[hook, from=6-4, to=6-5]
		\arrow[hook, from=7-4, to=7-5]
		\arrow[hook, from=8-4, to=8-5]
	\end{tikzcd}\]
	The rows for \(\mb{V}_\omega,\mb{W}_\omega\) are the same as for \(\mb{V},\mb{W}.\)
	Now we define derived versions of sites from Construction \ref{con:cinf_sites}. 
	\begin{definition}\label{def:der_sites}
		Let \(\mb{M}\) be any of the ten categories defined in Construction \ref{con:cinf_sites}. 
		\par Then its simplicial derived version \(\s\MM\) is the full subcategory of \(\s\CI\Alg_r\) spanned by algebras weakly equivalent to ones with \(\pi_0\) belonging to \(\mb{M}.\)
		\par Analogously we define DG derived version \(\dg\MM\) of \(\mb{M}\) by taking the full subcategory of \(\dg\CI\Alg_r\)  spanned by algebras weakly equivalent to ones with \(H_0\) belonging to \(\mb{M}.\)
		\par Finally, we can construct a DG semifree version \(\sfdg\MM\) of \(\mb{M}\) by taking the full subcategory of \(\dg\MM\) spanned by semifree algebras. 
	\end{definition}
	This definition clearly allows us to immediately promote the coverages of Construction~\ref{con:cinf_top} to homotopical coverages. Thus we obtain a higher version of each site. Hence there is a table analogous to the table \eqref{table_of_topoi} for model topoi of stacks, and we can promote all equivalences of the table \eqref{table_of_topoi} to equivalences of model topoi. 
	\begin{theorem}\label{thm:eq_for_models}
		Each site of the previous construction yields three homotopical sites: simplicial, differential graded, and semifree differential graded one. The categories of stacks over these three sites are all Quillen equivalent.
	\end{theorem}
	\begin{proof}
		Since the categories of \(\dg\MM\) and \(\s\MM\) are mapped to one another by the pair \((\Qc,\N)\) and this pair of functors preserves the homotopical coverages, Theorem \ref{thm:site_eq_gen} applies, and we get the desired equivalence. For the inclusion of the semifree site, we use the standard technique involving the deformation of the identity functor via the cofibrant replacement, see the book of Riehl \cite[\S 2]{riehl2014categorical} to see that there is an equivalence of homotopical categories. Subsequently we can apply the result of \TV \cite[Theorem 2.3.1]{toen2005homotopical}.
	\end{proof}
	\subsection{Deriving models for analysis}
	Obviously, all categories of Construction~\ref{con:cinf_sites} and coverages of Construction~\ref{con:cinf_top} can be defined for Examples \ref{ex:ag}, \ref{ex:hol} and even more generally for reduced Fermat theories (see Carchedi--Roytenberg \cite[Definition 2.28]{carchedi2012theories}). The difference between some of the categories can disappear. For example, for finitely generated algebras over \(\Com_K\) (\(K\) is a field) there is no difference between the categories \(\mb{V}\) and \(\mb{V}_\omega\) because all finitely generated \(\Com_K\)-algebras are Noetherian. It is also clear that Theorem \ref{thm:eq_for_models} can be adapted without any change to this setting if the Fermat theory in question is considered over a field of characteristic \(0.\)
	This idea can be formalized in the following result.
	\begin{theorem}\label{thm:eq_for_models_gen}
		Assume that we have a site \(\mb{M}\) consisting of finitely presented \(\Tc\)-algebras for a Fermat theory \(\Tc\). Then we can construct the three derived sites \(\dg\MM,\) \(\sfdg\MM,\) and \(\s\MM\) as in Definition \ref{def:der_sites}. If \(\Tc\) is a theory over a field of characteristic zero, these three sites produce equivalent model topoi of sheaves with the equivalence being induced by the pair \((\Qc,\N).\)
	\end{theorem}
	\begin{proof}
		The proof is essentially the same as the proof of Theorem \ref{thm:eq_for_models_gen}. We observe that the pair \((\Qc,\N)\) maps \(\dg\MM\) and \(\s\MM\) to one another and that homotopical coverages are preserved by construction. Hence Theorem \ref{thm:site_eq_gen} is applicable, and the result is proved. The equivalence with the topos of stacks on the semifree site is proved verbatim as in Theorem~\ref{thm:eq_for_models}.
	\end{proof}
	\subsection{Comparison with other models of derived manifolds}
	\subsubsection{DG (Q) manifolds of positive amplitude}\label{subsec:eq_dg}
	\begin{definition}[{\cite[Definition 1.15]{behrend2020derived}}]\label{def:g_pos_man}
		A {\it graded manifold} \(\M\) of {\it positive amplitude} is a pair \((M, \As)\), where \(M\) is a manifold,
		\[
		\As=\bigoplus_i\As^i
		\]
		is a sheaf of positively graded \(\CI(M)\)-algebras over \(M,\) such that there exists a positively graded vector space 
		\[
		V=V^1\oplus\ldots \oplus V^n
		\]
		and an open cover of \(M\) denoted by \(\{U_i\to M\}_{i\in I}\) such that there is an isomorphism of sheaves of \(\CI(M)\)-algebras
		\[
		\As\vert_{U_i}\cong \CI(U_i)\otimes \SS V^\vee,\quad \text{for all } i\in I.
		\]
	\end{definition}
	We assume that \(\M\) is finite-dimensional, i.e., both \(M\) and \(V\) have finite dimension. The following definition is originally due to Schwarz \cite{schwarz1993semiclassical} and Kontsevich; it was then expanded upon by \AKSZ\cite{alexandrov1997geometry}.
	\begin{definition}[{\cite[Definition 1.16]{behrend2020derived}}]
		A {\it differential graded manifold} (or {\it dg manifold}, or {\it Q-manifold}) is a triple \((M,\As,Q),\) where \((M,\As)\) is a graded manifold and \(Q\) is a derivation of the sheaf of \(\R\)-algebras \(\As\) such that \([Q,Q]=0.\) A dg manifold is said to be of {\it positive amplitude} if its underlying graded manifold is of positive amplitude.
	\end{definition}
	Morphisms between such objects are given by the following definition.
	\begin{definition}
		A {\it morphism} of dg manifolds \((M,\As,Q) \to (N, \Bs, Q')\) is a pair \((f,\Phi)\), where \(f : M \to N\) is a differentiable map of manifolds, and \(\Phi : \Bs \to f_* \As\) (or equivalently~\({\Phi: f^* \Bs \to \As}\)) is a morphism of sheaves of graded algebras such that \[Q \Phi = \Phi f^* Q'.\]
	\end{definition}
	\begin{definition}
		Denote by \(\dgMan\) the category of dg manifolds of positive amplitude and morphisms between them. Denote by \(\dgCart\) the full subcategory of \(\dgMan\) spanned by dg manifolds of the form \((\R^n,\As,Q)\) for all \(n\ge 0.\)
	\end{definition}
	\begin{construction}
		To a dg manifold \(\M=(M,\As,Q)\) of positive amplitude we associate its {\it algebra of smooth functions} \(\CI(\M)\in \dg\CI\Alg\) defined via formula
		\[
		\CI(\M)_0=\CI(M),\quad \CI(\M)_{\bullet \ge 1}=\Gamma(M;\As).
		\]
		The differential on \(\CI(\M)\) is induced by \(Q.\)
		The structure of the graded \(\CI\)-ring comes from the \(\CI(M)\)-algebra structure on global sections of \(\As.\)
		\par We obtain a functor
		\[
		\CI:\dgMan^{\op}\to \dg\CI\Alg.
		\]
	\end{construction}
	We have the following comparison result for Cartesian dg manifolds and the opposite category of finitely generated semifree DG \(\CI\)-rings.
	\begin{proposition}\label{pro:equiv_dg_Cart} Denote by \(\sfdg\CI\Alg_f\) the full subcategory of finitely generated semifree \(\CI\)-rings.
		There is an equivalence of categories between \(\dgCart\) and~\(\sfdg\CI\Alg_f^{\op}\) induced by \(\CI\) and its right adjoint \(\l\) (which we call the {\it dg locus functor}).
	\end{proposition}
	\begin{proof}
		This result is a straightforward generalization of the fact that the functor \(\CI(-)\) is a fully faithful embedding of the category of ordinary manifolds.
		We start by noting that \(\l\) can be defined as follows:
		\[
		\l(\A)=(l(\A_0),sh\A_{\ge 1}, d_\A),\quad sh\A_{\ge 1}(U):=\A_0\{\chi_U^{-1}\}\A_{\ge 1}.
		\]
		We denote \(\A_0\{\chi_U^{-1}\}\) the \(\CI\)-localization (see Borisov--Noel \cite[\S 2]{borisov2011simplicial}). Also \(l\) denotes the usual locus of an ordinary \(\CI\)-ring (see \cite[\S II.1]{moerdijk2013models}) and \(\chi_U\) is the function which is zero precisely on the complement of the open subset \(U.\)
		Denote by~\((\wt{M},\wt{\A},\wt{Q})\) the value of \(\l\CI\) on a dg manifold \((M,\A,Q).\) Then by the construction of \(\l\) we have canonical isomorphisms \(\wt{M}\cong M,\) also \(\wt{\As}\cong \As,\) and finally \(\wt{Q}\cong Q\) since it is uniquely defined by its action on global sections. The analogous isomorphisms for \(\CI\l\) are proved identically. 
	\end{proof}
	Proposition \ref{pro:equiv_dg_Cart} allows us to transfer the structure of a relative site on \(d\Cart_{\CI}\) to \(\dgCart.\) For this structure we have the following equivalence result.
	\begin{proposition}
		We have an equivalence of relative sites 
		\[
		\CI:\dgCart\to d\Cart_{\CI}
		\]
		inducing a Quillen equivalence on categories of higher sheaves (Definition \ref{def:sheaves}):
		\[
		\CI_!:\Sh(\dgCart)\rightleftarrows\Sh(d\Cart_{\CI}):(\CI)^*.
		\]
	\end{proposition}
	\begin{proof}
		By construction \(\CI\) reflects weak equivalences and is an equivalence of categories. Thus it induces an equivalence on Dwyer--Kan localizations. Also, by construction, the functor \(\CI\) reflects coverages, and thus it is an equivalence of relative sites (Definition~\ref{def:eq_rel_sites}). Hence Theorem \ref{thm:eq_sh_top} is applicable, and we have the claimed Quillen equivalence of sheaf categories.
	\end{proof}
	\subsubsection{Derived manifolds of Behrend--Liao--Xu}\label{subsec:eq_der}
	\begin{definition}[{\BLX \cite[Definition 1.1]{behrend2020derived}}]
		A {\it derived manifold} is a triple \(\M=(M,L,\lambda),\) where \(M\) is a manifold (called the {\it base}), \(L\) is a (finite-dimensional) graded vector bundle
		\[
		L=L^1\oplus \ldots \oplus L^n
		\]
		over \(M,\) and \(\lambda=(\lambda_k)_{k\ge 0}\) is a sequence of multi-linear operations of degree \(+1\)
		\[
		\lambda_k:L\times_M\ldots\times_M L\to L,\quad k\ge 0.
		\]
		The operations \(\lambda_k\) are required to be smooth maps over \(M\) and to induce the structure of {\it curved \(L_\infty[1]\)-algebra} (see \cite[Appendix B]{behrend2020derived}) on each fiber. 
	\end{definition}
	\begin{definition}[{\BLX \cite[Definition 1.2]{behrend2020derived}}]\label{def:CL}
		The zero locus of the map \(\lambda_0:M\to L\) is called the {\it classical locus} of \(\M=(M,L,\lambda).\) We denote it by \(\CL(\M)\sse\M.\) The points of \(\CL(\M)\) are called {\it classical points} of \(\M,\) we denote this set by \(\M(\R).\) We consider \(\CL(\M)\) not as a mere set, but as a \(\CI\)-locus (i.e., the formal dual of the corresponding \(\CI\)-ring). 
	\end{definition}
	\begin{definition}[{\BLX \cite[Definition 1.3]{behrend2020derived}}]
		A {\it morphism} of derived manifolds \[{(M,L,\lambda)\xrightarrow{(f,\phi)} (N,K,\mu)}\] is a pair consisting of a smooth map \(f:M\to N,\) and a covering morphism of graded bundles \(\phi:L\to K\) inducing a morphism of curved \(L_\infty[1]\)-algebras on the fibers. 
	\end{definition}
	\begin{definition}[{\BLX \cite[Definition 1.9]{behrend2020derived}}]
		Let \(\M=(M,L,\lambda)\) be a derived manifold. Then for a classical point \(P\in M\) we have \(\lambda_0(P)=0\) and hence there is a natural splitting 
		\[
		T_{\lambda_0(P)}L^1=T_P M\oplus L^1\vert_P.
		\]
		Denote by \(D_P\lambda_0\) the composite map
		\[
		D_P\lambda_0:=\left[T_P M\xrightarrow{T_P\lambda_0} T_P M\oplus L^1\vert_P\xrightarrow{\pi_2} L^1\vert_P\right].
		\]
		For a classical point \(P\) in \(\M\) we can now form the {\it tangent complex} of \(\M\) at \(P:\)
		\[
		T_P\M=\left[\underbrace{T_P M}_{\deg=0}\xrightarrow{D_P\lambda_0} \underbrace{L^1\vert_P}_{\deg=1}\xrightarrow{\lambda_1\vert_P} \underbrace{L^2\vert_P}_{\deg=2}\xrightarrow{\lambda_1\vert_P}\ldots\right]
		\]
	\end{definition}
	\begin{definition}[{\BLX \cite[Definition 1.12]{behrend2020derived}}]\label{def:we_dMan}
		A morphism of derived manifolds \(\M\to \Ns\) is a {\it weak equivalence} if it induces a bijection on classical loci and quasi-isomorphisms of tangent complexes over each classical point. 
	\end{definition}
	\begin{definition}[{\BLX \cite[Definition 1.21]{behrend2020derived}}]\label{def:fib_dMan}
		A morphism of derived manifolds \[(M,L,\lambda)\to (N,K,\mu)\] is a {\it fibration} if it induces a submersion \(M\to N\) and a degree-wise surjection of graded bundles \(L\to K.\)
	\end{definition}
	\begin{construction}[{\BLX \cite[\S 1.2]{behrend2020derived}}]
		Given a derived manifold \(\M=(M,L,\lambda)\) the {\it algebra of functions} is an object \(\CI(\M)\) of \(\dg\CI\Alg\) defined as the algebra of global sections of the sheaf of differential graded algebras
		\[
		\CI(\M)=\Gamma(M;\SS_{\CI(M)} L^\vee).
		\] 
		The grading is induced by the grading on \(L\) and the component \(\CI(\M)_0\) clearly coincides with \(\CI(M),\) giving \(\CI(\M)\) the structure of a graded \(\CI\)-ring. The map \(\lambda\) induces a differential, making \(\CI(M)\) into an object of \(\dg\CI\Alg.\)
	\end{construction}
	We have the following slight improvement of a result due to \BLX {\cite[Proposition 1.32]{behrend2020derived}}. 
	\begin{proposition}\label{pro:iso_on_loci}
		Let \(\A,\B\) be two DG \(\CI\)-rings. Assume \(\A=\CI(\M)\) and \(\B=\CI(\Ns)\) for some derived manifolds \(\M,\Ns.\) Let \(f:\M\to \Ns\) be a morphism of derived manifolds inducing a quasi-isomorphism \(f^*:\B\to\A.\) Then \(f\)~induces an isomorphism of classical loci (see Definition \ref{def:CL}). This is an isomorphism of \(\CI\)-loci. Moreover, the set of classical points of a derived manifold \(\M\) is canonically isomorphic to the set of \(\dg\CI\Alg\)-morphisms into the initial algebra:
		\[
		\M(\R)\cong\Hom_{\dg\CI\Alg}(\CI(\M);\R).
		\]
	\end{proposition}
	\begin{proof}
		By construction \(\CL(\M)\) is dual to the \(\CI\)-ring \(H_0(\CI(\M))\) and hence the quasi-isomorphism, which must induce an isomorphism on \(H_0,\) induces an isomorphism of classical loci. 
		\par The second statement is clear from the definition of \(\CL(\M)\) as the locus of \(H_0(\M).\) Indeed, all \(\dg\CI\Alg\)-morphisms from \(\CI(\M)\) to the initial algebra \(\R\) bijectively correspond to homomorphisms of commutative algebras \(\CI(\M)\to \R\) which annihilate the image of the differential. 
	\end{proof}
	\begin{notation}
		Denote by \(\dMan\) the category of derived manifolds with distinguished classes of weak equivalences and fibrations defined above. Denote by \(\dCart\) the full subcategory of \(\dMan\) spanned by the derived manifolds with a Cartesian base.
	\end{notation}
	
	\begin{lemma}\label{lem:alt_tangent}
		Let \(\M\) be a derived manifold with a Cartesian base. In a classical point~\({P:\CI(\M)\to\R}\) there is a canonical isomorphism of chain complexes
		\[
		T_P\M\cong (\ker P/(\ker P)^2)^\vee.
		\] 
		Here \(\vee\) stands for the linear dual. 
	\end{lemma}
	\begin{proof}
		Let \(\M=(M,L,\lambda)\) be a derived manifold with a Cartesian base. 
		Since any bundle over a Cartesian space is trivial (graded or otherwise). There is an isomorphism of graded algebras:
		\[
		\CI(\M)\cong \CI(M)\otimes \SS(L\vert_P^\vee).
		\]
		Here the grading of the generators in the symmetric algebra is given by the grading on \(L.\) This isomorphism immediately induces an isomorphism of differential graded \(\CI\)-rings, where the differential on the right is given by \(\lambda.\)
		\par Now the isomorphism becomes quite clear because in degree \(0\), this is just the usual description of the tangent space to a Cartesian space and in higher degrees \((\ker P/(\ker P)^2)^\vee\) coincides with the graded vector space of fibers at \(P\) of the bundle \(L.\) The differential on~\(T_P\M\), and \((\ker P/(\ker P)^2)^\vee\) is the same because the structure of the differential graded algebra on~\({\CI(M)\otimes \SS(L\vert_P^\vee)}\) was induced by \(\lambda.\)
	\end{proof}
	\begin{remark}
		Lemma \ref{lem:alt_tangent} can be extended to an arbitrary derived manifold. However, in this case, it is necessary to localize the ring \(\CI(\M)\) to a Cartesian neighbourhood. 
	\end{remark}
	\begin{proposition}[{\cite[Proposition 1.18]{behrend2020derived}}]
		There is an equivalence of the categories between dg manifolds of positive amplitude and derived manifolds:
		\[
		\omega:\dMan\to \dgMan.
		\]
	\end{proposition}
	\begin{corollary}\label{cor:sf_dCart_eq}
		The functor \(\CI(-):\dMan\to \dg\CI\Alg^{\op}\) restricted to the subcategory \(\dCart\) takes values in \(\sfdg\CI\Alg.\) Moreover, it admits a right adjoint \[{\dl:\sfdg\CI\Alg^{\op}_r\to \dCart}\] which can be computed as the composite functor: \(\dl=\l\omega\) of the equivalence \(\omega\) and dg locus functor of Proposition \ref{pro:equiv_dg_Cart}.  
	\end{corollary}
	\begin{proof}
		Clearly, \(\omega\) maps derived manifolds with a Cartesian base to dg manifolds with a Cartesian base. The rest follows from Proposition \ref{pro:equiv_dg_Cart}.
	\end{proof}
	By Behrend--Liao--Xu \cite[Theorem 2.15]{behrend2020derived} the category \(\dMan\) forms a category of fibrant objects in the sense of Goerss--Jardine \cite[\S I.9]{goerss2009simplicial} (originally the concept was introduced by Brown \cite{brown1973abstract}). The category \(\dCart\) is not stable under taking pullbacks, but still carries distinguished classes of weak equivalences and fibrations (see Definitions \ref{def:we_dMan}, \ref{def:fib_dMan}).  
	\par Clearly, the category \(\sfdg\CI\Alg^{\op}_f\) also carries such a structure. Indeed, it is the dual of the full subcategory spanned by cofibrant objects in the model category~\(\dg\CI\Alg.\)
	\par We have the following comparison result:
	\begin{theorem}\label{thm:comparison_for_dCart}
		The equivalence of categories 
		\[
		\sfdg\CI\Alg^{\op}_f\xrightarrow{\dl}\dCart
		\]
		reflects {\it all} weak equivalences, not only the acyclic fibrations. Additionally, \(\dl\) reflects fibrations. Thus \(\dl\) is a Dwyer--Kan equivalence of relative categories. 
	\end{theorem}
	\begin{proof} 
		\par We start by showing that the functor \(\CI(-)\) reflects weak equivalences. Observe that each weak equivalence in \(\dg\CI\Alg\) between algebras of functions on derived manifolds induces an isomorphism of classical loci by Proposition \ref{pro:iso_on_loci}. Let \(\alpha:\A\to \B\) be a quasi-isomorphism of \(\sfdg\CI\Alg_f.\) Consider a point \(P\) of the classical locus \(\CL(\dl(\A))\) and the corresponding point \(Q\) of the classical locus \(\CL(\dl(\B)).\) Then by Lemma \ref{lem:alt_tangent} the cotangent complex of \(\dl(\A)\) at \(P\) can expressed as a module of K\"{a}hler differentials
		\[
		T_P^*\dl(\A)=\Omega_{K}^1(\A)=\ker P/(\ker P)^2.
		\]
		Since \(\A\) is cofibrant and \(\Omega^1_K\) is a left Quillen functor, we have a quasi-isomorphism of cotangent complexes:
		\[
		T^*_P\dl(\A)\simeq T^*_Q\dl(\B).
		\] 
		Thus weak equivalences in \(\sfdg\CI\Alg_f^{\op}\) are mapped to weak equivalences between Cartesian derived manifolds in the sense of Definition \ref{def:we_dMan}.
		Now a theorem of Behrend--Liao--Xu \cite[Theorem 3.13]{behrend2020derived} implies that weak equivalences of derived Cartesian spaces correspond to quasi-isomorphisms on their algebras of functions.  
		\par By construction, fibrations in \(\sfdg\CI\Alg_f^{\op}\) correspond to global projections (semifree extensions of function algebras), which are mapped to submersions of derived Cartesian spaces. Hence \(\dl\) indeed preserves fibrations. 
	\end{proof}
	Theorem \ref{thm:comparison_for_dCart} allows us to positively resolve the conjecture of Behrend--Liao-Xu~\cite[p. 6]{behrend2020derived}. While we were preparing this paper for uploading to arXiv an alternative proof of Corollary \ref{cor:conj} was given by Carchedi \cite[Theorem 5.20]{carchedi2023derived}.
	\begin{corollary}\label{cor:conj}
		A morphism of derived manifolds is a weak equivalence if and only if it induces a quasi-isomorphism on algebras of functions.
	\end{corollary}
	\begin{proof}
		By Theorem \ref{thm:comparison_for_dCart}, the functor \(\CI(-)\) reflects weak equivalences on derived Cartesian spaces. The condition of Definition \ref{def:we_dMan} for a map to be a weak equivalence of derived manifolds is clearly local. Since any derived manifold admits a covering by derived Cartesian spaces open at the level of the base manifold, the result follows. 
	\end{proof}
	\begin{theorem}\label{thm:eq_dMan}
		The relative site \((\dMan, \W)\) of derived manifolds and weak equivalences defines a topos of higher sheaves that is Quillen equivalent to the topos of higher sheaves over \(\sfdg\CI\Alg_f^{\op}.\)
	\end{theorem}
	\begin{proof}
		There is an equivalence of relative sites \((\dMan,\W)\to (\dCart,\W)\) induced by the fibrant replacement functor in \(\dg\CI\Alg^{\op}.\) It remains to apply Theorem \ref{thm:comparison_for_dCart} to see that the functor \(\dl\) is an equivalence of relative sites. Hence, we have an equivalence of sheaf topoi by Theorem \ref{thm:eq_sh_top}.
	\end{proof}
	\subsection{Derived differential cohomology}\label{sec:dif_coh}
	In this section, we define differential cohomology of derived differentiable stacks. For a general treatment of differential forms on derived algebraic stacks see Calaque--Pantev--To\"{e}n--Vaqui\'{e}--Vezzosi \cite[\S 1.3.2, 1.4.1]{calaque2017shifted}.
	\begin{definition}\label{def:KD}
		Let \(\Rc=l(\A)\) be an object of \(d\Cart_{\CI}\) corresponding to \({\A\in \sfdg\CI\Alg.}\) Then we define the module of K\"{a}hler differentials \(\Omega_{K}^1(\Rc)\) of \(\Rc\) as the value of the left adjoint to the forgetful functor
		\[
		\UU:\Ab(\dg\CI\Alg/\A)\to \dg\CI\Alg
		\]
		on the object \(\A.\)
	\end{definition}
	Note that the definition above is just the standard construction of K\"{a}hler differentials in terms of Beck modules (see, e.g., Quillen \cite{quillen1970co}). One can also give a more explicit version of the same definition following Dubuc--Kock \cite[\S 2]{dubuc19841}.
	\begin{definition}\label{def:KD_der}
		Let \(\Rc=l(\A)\) be an object of \(d\Cart_{\CI}\) corresponding to \({\A\in \sfdg\CI\Alg.}\)	Let \(N\) be a differential graded module over the algebra \(\A\) (regarded as a commutative dg algebra). Then a map \(\delta:A\to N\) is a \(\CI\)-\emph{derivation} if it is a derivation in the usual graded commutative sense and also for any smooth \(\phi:\R^n\to \R\) and any collection \(\alpha_1,\ldots,\alpha_n\in\A_0\) we have
		\[
			\delta(\phi(\alpha_1,\ldots,\alpha_n))=\sum_{i=1}^n\frac{\pd\phi}{\pd x_i}(\alpha_1,\ldots,\alpha_n)\cdot \delta(\alpha_i).
		\]
		We denote by \(\Der_\A(N)\) the module of \(\CI\)-derivations of \(\A\) with values in \(N.\) The module of Kahler differentials of \(\Rc\) is identified with the object in \(\Ab(\dg\CI\Alg/\A)\) which corepresents the following functor:
		\[
			\dg\CI\Alg/\A\ni N\mapsto \Der_\A(N).
		\]
	\end{definition}
	\begin{lemma}
		Constructions of definitions \ref{def:KD} and \ref{def:KD_der} yield canonically isomorphic object.
	\end{lemma}
	\begin{proof}
		Let \(\A\) be a dual to an object \(\Rc\) of \(d\Cart_{\CI}.\)	Then the category \(\Ab(\dg\CI\Alg/\A)\) forms a subcategory of the category of the category \(\Ab(\dg\Com\Alg/\A).\) Moreover, each module \(M\) over \(\A\) is naturally a module over \(\A_0.\) By a result of Dubuc--Kock \cite[Proposition 2.2]{dubuc19841} \(\CI\)-derivations of \(\A_0\) valued in \(M\) could be identified with morphisms from \(\Omega^1_K(\A_0^{\op}).\) It remains to note that this natural bijection extends to derivations of \(\A\) because any \(\CI\)-derivation of \(\A\) valued in \(M\) is by definition a graded-commutative derivation which restricts to a \(\CI\)-derivation of \(\A_0.\)
	\end{proof}
	\begin{definition}\label{def:dif_forms}
		Let \(\Rc=l(\A)\) be an object of \(d\Cart_{\CI}\) corresponding to \({\A\in \sfdg\CI\Alg.}\) Then we define the cotangent complex of \(\Rc\) as the left derived functor of the functor of K\"{a}hler differentials (Definition \ref{def:KD}):
		\begin{equation}\label{eq:forms}
			L_{\Rc}=\LL\Omega_{K}^1\Rc.
		\end{equation}
		This definition then immediately gives a presheaf on \(d\Cart_{\CI}.\) Note that because the site \(d\Cart_{\CI}\) consists of cofibrant objects the total left derived functor in formula \eqref{eq:forms} is actually not needed. The object \(L_\Rc\) actually coincides with the complex of K\"{a}hler differentials \(\Omega^1_K\Rc.\) 
	\end{definition}
	\begin{construction}\label{con:int_dR}
		Clearly, Definition \ref{def:dif_forms} and its homotopy invariance allows us to define sheaves of differential forms for any object in the higher sheaf topos \(\Sh(d\Cart_{\CI})\) (see Definition~\ref{def:sheaves}). Indeed, we can just apply the symmetric algebra functor to the functor \(L_\Rc\) shifted by \(-1\) thus obtaining the presheaf \(\Omega^*.\) Taking the derived mapping space into \(\Omega^*\) defines a sheaf of differential forms on \(\Sh(d\Cart_{\CI})\). Thus the result is the \emph{de Rham complex of stacks}:
		\begin{equation}\label{eq:dR}
			\Omega^*=\left[\underbrace{\left(\CI(-)=\Omega^0\right)}_{\deg=0}\xrightarrow{d}\underbrace{\left(L_{(-)}=\Omega^1\right)}_{\deg=-1}\xrightarrow{d}\underbrace{\left(\SS^2\Omega^1=\Omega^2\right)}_{\deg=-2}\to\ldots\right] 
		\end{equation}
	\end{construction}
	\par The construction above is a direct extension of a specialization of the definition given by Pridham in \cite[Definition 1.4]{pridham2018outline} for the \emph{superdegree zero case}, i.e. for super NQ manifolds with function algebras consisting only of even elements. 
	\par Now we define the Deligne complex. For a reference on the ordinary Deligne complex and its relation to classical differential cohomology, see~\cite[\S 7.3]{amabel2021differential}.
	\begin{definition}[Deligne complex]\label{def:dif_coh}
		Denote by \(\Z\) the locally constant sheaf of \(\Z\)-valued functions on~\(d\Cart_{\CI}.\)
		For \(k\ge 1\) consider a complex of sheaves of chain complexes (i.e. a sheaf bicomplex):
		\begin{gather}\label{eq:Deligne_complex}
		\D^k=\left[\underbrace{\Z}_{\deg=0}\xhookrightarrow{2\pi i} \underbrace{\Omega^0}_{\deg=-1}\xrightarrow{d} \underbrace{\Omega^1}_{\deg=-2}\xrightarrow{d} \underbrace{\Omega^2}_{\deg=-3}\xrightarrow{d} \ldots\xrightarrow{d} \underbrace{\Omega^{k-1}}_{\deg=-k}\xrightarrow{d} 0\to 0\to \ldots\right].
		\end{gather}
		Now we define the {\it \(k\)-th differential cohomology group} \(\hat{H}^k(X)\) as 
		\begin{equation}\label{equation:dif_coh_as_coh_of_a_space}
			\hat{H}^k(X)=H^k(\operatorname{Tot}[X,\D^k],d).
		\end{equation}
		Here \([-,-]\) is the set of homotopy classes of maps, while \(\operatorname{Tot}\) stands for the totalization complex obtained by taking the sum of the exterior differential and the differential of the cotangent complex.
	\end{definition}
	\begin{remark} Note that the differential cohomology of a stack \(X\) admits a slightly different description in terms of the derived mapping space. Namely, \(\hat{H}^k(X)=H^k(\ul{\Hom}(X,\mc{D})),\) where the right hand side is the derived mapping space between the two stacks.
	\end{remark} 
	When compared with the approach of Calaque--Pantev--To\"{e}n--Vaqui\'{e}--Vezzosi \cite{calaque2017shifted} our construction \eqref{eq:Deligne_complex} recovers the Deligne version of the \emph{internal} de Rham complex, while construction \eqref{equation:dif_coh_as_coh_of_a_space} calculates cohomology of the \emph{absolute} Deligne complex (see Calaque--Pantev--To\"{e}n--Vaqui\'{e}--Vezzosi \cite[Proposition 1.3.8, Definition 1.3.9]{calaque2017shifted}). Indeed, the de Rham complex \eqref{eq:dR} of Construction \ref{con:int_dR} is just a left adjoint of the functor which takes the degree \(0\) subalgebra in differential bigraded \(\CI\)-algebras in sheaves on \(d\Cart_{\CI}.\) Hence the resulting complex in the category of sheaves is the \(\CI\)-version of the internal de Rham complex of Calaque--Pantev--To\"{e}n--Vaqui\'{e}--Vezzosi.
	\begin{remark}
		Note that by the above the Deligne complex is very different from being a mere free algebra on degree \(1\) differential forms. Indeed, it carries two differentials, one of which is the exterior differential, while the other one is the internal differential of the cotangent complex. Hence differential cohomology is a non-trivial combination of the two quite different in spirit kinds of data.
	\end{remark}
	\begin{remark}
		Since all presheaves except \(\Z\) in \(\D^k\) are fibrant, Definition \ref{def:dif_coh} is homotopy invariant. That is, we don't need to make a fibrant replacement on the nose or a cofibrant replacement on the tail in \([X,\D^k]\) to obtain \(\pi_0\) of a derived mapping space. 
	\end{remark}
	\begin{remark}
		Clearly Definitions \ref{def:KD}-\ref{def:dif_forms} apply to arbitrary Fermat theories and hence Definition \ref{def:dif_coh} makes sense for any Fermat theory as well.
	\end{remark}
	

\begin{thebibliography}{10}
		
		\bibitem{alexandrov1997geometry}
		Mikhail Alexandrov, Maxim Kontsevich, Albert Schwarz, and Oleg Zaboronsky. The geometry of the master equation and topological quantum field theory. International Journal of Modern Physics A (1997), 12:7, 1405--1429. \arXiv{hep-th/9502010}, \doi{10.1142/s0217751x97001031}.
		
		\bibitem{amabel2021differential}
		Araminta Amabel, Arun Debray, and Peter~J.~Haine. Differential cohomology: categories, characteristic classes, and connections. \arXiv{2109.12250}.
		
		\bibitem{andre2013homologie}
		Michel Andr{\'e}. Homologie des algebres commutatives. Volume~206, (Springer, 1974). \doi{10.1007/978-3-642-51449-4}.
		
		\bibitem{artin1973theorie}
		Michael Artin, Alexandre Grothendieck, and Jean~Louis Verdier. Th{\'e}orie des topos et cohomologie etale des sch{\'e}mas: tome 3, (Springer, 1973). \doi{10.1007/BFb0061319}.
		
		\bibitem{artin2006etale}
		Michael Artin and Barry Mazur. Etale homotopy. Volume~100, (Springer, 2006). \doi{10.1007/bfb0080957}.
		
		\bibitem{behrend2020derived}
		Kai Behrend, Hsuan-Yi Liao, and Ping Xu. Derived differentiable manifolds. \arXiv{2006.01376}.
		
		\bibitem{borisov2011simplicial}
		Dennis Borisov and Justin Noel. Simplicial approach to derived differential manifolds.
		\arXiv{1112.0033}.
		
		\bibitem{brown1973abstract}
		Kenneth~S.~Brown. Abstract homotopy theory and generalized sheaf cohomology. Transactions of the American Mathematical Society (1973), 186, 419--458. \doi{10.1090/s0002-9947-1973-0341469-9}.
		
		\bibitem{calaque2017shifted}
		Damien Calaque, Tony Pantev, Bertrand To{\"e}n, Michel Vaqui{\'e}, and Gabriele
		Vezzosi. Shifted poisson structures and deformation quantization. Journal of topology (2017), 10:2, 483--584. \arXiv{1506.03699}, \doi{10.1112/topo.12012}.
		
		\bibitem{carchedi2023derived}
		David Carchedi. Derived manifolds as differential graded manifolds. \arXiv{2303.11140}.
 		
		\bibitem{carchedi2012homological}
		David Carchedi and Dmitry Roytenberg. Homological algebra for superalgebras of differentiable functions. \arXiv{1212.3745}.
		
		\bibitem{carchedi2012theories}
		David Carchedi and Dmitry Roytenberg. On theories of superalgebras of differentiable functions. \arXiv{1211.6134}.
		
		\bibitem{dubuc1981c}
		Eduardo~J.~Dubuc. $\CI$-schemes. American Journal of Mathematics (1981), 103:4, 683--690. \doi{10.2307/2374046}.
		
		\bibitem{dubuc19841}
		Eduardo~J.~Dubuc and Anders Kock. On 1-form classifiers. Communications in Algebra (1984), 12:12, 1471--1531. \doi{10.1080/00927878408823064}.
		
		\bibitem{dugger2004hypercovers}
		Daniel Dugger, Sharon Hollander, and Daniel~C.~Isaksen. Hypercovers and simplicial presheaves. In Mathematical Proceedings of the Cambridge Philosophical Society, Volume~136, (Cambridge University Press, 2004), 9--51. \arXiv{math/0205027}, \doi{10.1017/s0305004103007175}.
		
		\bibitem{dwyer1987equivalences}
		William~G.~Dwyer and Daniel~M.~Kan. Equivalences between homotopy theories of diagrams. Algebraic topology and algebraic K-theory, Annals of Math. Studies (1987), 113, 180--205. \doi{10.1515/9781400882113-009}.
		
		\bibitem{dwyer2004homotopy}
		William~G.~Dwyer, Philip~S.~Hirschhorn, Daniel~M.~Kan, and Jeffrey~H.~Smith. Homotopy limit functors on model categories and homotopical categories. Number~113, (AMS, 2004). \doi{10.1090/surv/113}.
		
		\bibitem{dwyer1980simplicial}
		William~G.~Dwyer and Daniel~M.~Kan. Simplicial localizations of categories. Journal of Pure and Applied Algebra (1980), 17:3, 267--284. \doi{10.1016/0022-4049(80)90049-3}.
		
		\bibitem{gelfand2013methods}
		Sergei~I.~Gelfand and Yuri~I.~Manin. Methods of homological algebra (Springer, 1996). \doi{10.1007/978-3-662-03220-6}.
		
		\bibitem{goerss2009simplicial}
		Paul~G.~Goerss and John~F.~Jardine. Simplicial homotopy theory (Springer, 2009). \doi{10.1007/978-3-0346-0189-4}.
		
		\bibitem{gwilliam2018enhancing}
		Owen Gwilliam and Dmitri Pavlov. Enhancing the filtered derived category. Journal of Pure and Applied Algebra (2018), 222:11, 3621--3674. \arXiv{1602.01515}, \doi{10.1016/j.jpaa.2018.01.004}.
		
		\bibitem{hirschhorn2003model}
		Philip~S.~Hirschhorn. Model categories and their localizations. Number~99, (AMS, 2003). \doi{10.1090/surv/099}.
		
		\bibitem{hovey2007model}
		Mark Hovey. Model categories, Number~63, (AMS, 1999). \doi{10.1090/surv/063}. 
		
		\bibitem{jardine2006intermediate}
		J.~F.~Jardine. Intermediate model structures for simplicial presheaves. Canadian Mathematical Bulletin (2006), 49:3, 407--413. \doi{10.4153/cmb-2006-040-8}.
		
		\bibitem{lurie2004derived}
		Jacob Lurie. Derived algebraic geometry. PhD thesis, MIT, 2004. \url{http://hdl.handle.net/1721.1/30144}.
		
		\bibitem{lurie2009higher}
		Jacob Lurie. Higher topos theory, (Princeton University Press, 2009). \doi{10.1515/9781400830558}.
		
		\bibitem{maclane2012sheaves}
		Saunders MacLane and Ieke Moerdijk. Sheaves in geometry and logic: A first introduction to topos theory, (Springer, 1994). \doi{10.1007/978-1-4612-0927-0}.
		
		\bibitem{moerdijk2013models}
		Ieke Moerdijk and Gonzalo~E.~Reyes. Models for smooth infinitesimal analysis, (Springer, 2013). \doi{10.1007/978-1-4757-4143-8}.
		
		\bibitem{nuiten2018lie}
		J.~Nuiten. Lie algebroids in derived differential topology. PhD thesis, Utrecht University (2018). \url{https://imag.umontpellier.fr/~nuiten/Writing/LieAlgdDerDiffTopScreen.pdf}. 

		\bibitem{pridham2018outline}
		J. P.~Pridham. An Outline of Shifted Poisson Structures and Deformation Quantisation in Derived Differential Geometry, 2018. \arXiv{1804.07622}.

		
		\bibitem{pridham2020differential}
		J.P.~Pridham. A differential graded model for derived analytic geometry. Advances in Mathematics (2020), 360, 106922. \arXiv{1805.08538}, \doi{10.1016/j.aim.2019.106922}.
		
		\bibitem{quillen1968homology}
		Daniel Quillen.
		\newblock Homology of commutative rings, mimeographed notes (1968). Unpublished.			\url{https://web.archive.org/web/20150420184538/http://chromotopy.org/paste/quillen.djvu}.
		
		\bibitem{quillen1969rational}
		Daniel Quillen. Rational homotopy theory. Annals of Mathematics (1969), 205--295. \doi{10.2307/1970725}.
		
		\bibitem{quillen1970co}
		Daniel Quillen. On the (co-)homology of commutative rings. In Proc. Symp. Pure Math, volume~17 (1970), 65--87. \doi{10.1090/pspum/017/0257068}.
		
		\bibitem{quillen2006homotopical}
		Daniel Quillen. Homotopical algebra, volume~43 (Springer, 2006). \doi{10.1007/bfb0097438}.
		
		\bibitem{riehl2014categorical}
		Emily Riehl. Categorical homotopy theory, volume~24 (Cambridge University Press, 2014). \doi{10.1017/CBO9781107261457}.
		
		\bibitem{schwarz1993semiclassical}
		Albert Schwarz. Semiclassical approximation in Batalin-Vilkovisky formalism.
		Communications in Mathematical Physics, 158 (1993), 373--396. \arXiv{hep-th/9210115}, \doi{10.1007/BF02108080}. 
		
		\bibitem{spivak2010derived}
		David~I.~Spivak. Derived smooth manifolds.
		Duke Mathematical Journal, 153:1 (2010), 55--128. \arXiv{0810.5174}, \doi{10.1215/00127094-2010-021}.
		
		\bibitem{tierney1976spectrum}
		Myles Tierney. On the spectrum of a ringed topos.
		In "Algebra, Topology, and Category Theory", {\it A Collection of Papers in Honor of Samuel Eilenberg} (Elsevier, 1976), 189--210. \doi{10.1016/C2013-0-10841-0}.
		
		\bibitem{toen2005homotopical}
		Bertrand To{\"e}n and Gabriele Vezzosi.
		Homotopical algebraic geometry I: Topos theory.
		Advances in mathematics 193:2 (2005), 257--372. \arXiv{math/0207028}, \doi{10.1016/j.aim.2004.05.004}.
		
		\bibitem{vezzosi2011derived}
		Gabriele Vezzosi. Derived critical loci I-basics.
		\arXiv{1109.5213}.
		
	\end{thebibliography}
\end{document}